\documentclass[11pt,letterpaper]{article}

\pdfoutput=1

\usepackage{graphics,epsfig,graphicx,color}
\usepackage[caption = false]{subfig}
\usepackage{float}
\usepackage{placeins}
\usepackage{capt-of}

\usepackage{amsmath}
\usepackage{mathtools}
\numberwithin{equation}{section}

\usepackage{amssymb,latexsym}

\usepackage{setspace}

\usepackage{mathrsfs,amsfonts}

\usepackage{amsthm,amsxtra}

\usepackage[final]{showkeys}

\usepackage{stmaryrd}

\usepackage{algorithm}
\usepackage{algorithmicx}
\usepackage{algpseudocode}

\usepackage[openbib]{currvita}
\usepackage[square,numbers,sort]{natbib}

\usepackage{etaremune}
\usepackage{enumitem} 
\usepackage{physics2}
    \usephysicsmodule{ab, ab.legacy}

\newif\ifPDF
\ifx\pdfoutput\undefined
	\PDFfalse
\else
	\ifnum\pdfoutput > 0
		\PDFtrue
	\else
		\PDFfalse
	\fi
\fi

\ifPDF
	\usepackage[debug,pdftex,colorlinks=true, 
	linkcolor=blue, bookmarksopen=false,
	plainpages=false,pdfpagelabels]{hyperref}
\else
	\usepackage[dvips]{hyperref}
\fi

\usepackage{fancyhdr}

\usepackage{comment}

\usepackage[toc,page]{appendix}

\usepackage{cancel}

\usepackage{multirow}
\usepackage{tabularx}
\usepackage{booktabs}

\usepackage{tcolorbox}

\newtheorem{theorem}{Theorem}[section]
\newtheorem{lemma}[theorem]{Lemma}

\newtheorem{remark}[theorem]{Remark}

\newtheorem{truth}[theorem]{Theorem}

\newcommand{\pp}{\partial}

\DeclareMathOperator*{\esssup}{ess\,sup}

\renewcommand{\vec}{\mathbf} 

\newcommand{\dbrack}[1]{[\!\![#1]\!\!]} 
\newcommand{\dbrace}[1]{\{\!\!\{#1\}\!\!\}} 

\DeclareMathAlphabet{\mathbbold}{U}{bbold}{m}{n} 
\newcommand{\RNum}[1]{\uppercase\expandafter{\romannumeral #1\relax}} 

\newenvironment{keywords}
{\noindent{\bf Key words.}\small}{\par\vspace{1ex}}
\newenvironment{AMS}
{\noindent{\bf AMS subject classifications 2020.}\small}{\par}

\newcommand{\DELETE}[1]{}

\DeclarePairedDelimiterX\setc[2]{\{}{\}}{\,#1 \;\delimsize\vert\; #2\,} 

\title{An Energy-based discontinuous Galerkin method for the Westervelt equation with absorbing boundary conditions}

\author{Yiwei Wang \thanks{Department of Computational Applied Mathematics and Operations Research, Rice University, Houston, TX, 77005, USA. Email: yw253@rice.edu} \quad\quad  Lu Zhang\thanks{Department of Computational Applied Mathematics and Operations Research, Rice University, Houston, TX, 77005, USA.  Email: lz82@rice.edu}
} 

\begin{document}

\maketitle

\begin{abstract}
This paper develops and analyzes an energy-based discontinuous Galerkin method for the Westervelt equation in the acoustic-velocity-potential formulation with spatially varying material parameters and nonlinear absorbing boundary conditions. The semi-discrete scheme is constructed from the underlying energy balance and incorporates an additional consistent term to accommodate the strong damping operator within the discrete energy framework. Numerical fluxes are designed at both interelement and physical boundaries to provide a unified treatment of interior and boundary dissipation. For the nonlinear analysis, we study an auxiliary variable-coefficient linearized problem and establish stability and a priori error estimates, including the additional terms induced by the solution-dependent absorbing boundary flux. These estimates are combined with a fixed-point argument to prove existence, uniqueness, and convergence of the nonlinear semi-discrete approximation under suitable regularity, smallness, and nondegeneracy assumptions. Numerical experiments confirm the theoretical convergence results and illustrate the performance of the proposed method.
\end{abstract}
 




\section{Introduction}
Nonlinear acoustic wave propagation arises in a broad range of scientific and engineering applications, including medical ultrasound, high-intensity focused ultrasound, lithotripsy, nondestructive testing, and acoustic imaging; see \cite{HamiltonBlackstock1998,Duck2002,mo2012ultrasound, fierro2015nonlinear,jhang2009nonlinear,matlack2012evaluation,
wu2003preliminary} and the references therein. At sufficiently large acoustic amplitudes, nonlinear effects such as harmonic generation, waveform distortion, and wave steepening become significant, making nonlinear acoustic models necessary for an accurate description of finite-amplitude wave propagation. In many applications, the acoustic waves propagate through a region that is much larger than the computational domain of interest. Numerical simulations therefore require an artificial truncation of the propagation region, and an inappropriate treatment of the resulting computational boundary may produce spurious reflections that travel back into the domain and contaminate the computed acoustic field. Absorbing boundary conditions are introduced to mitigate such artificial reflections by allowing outgoing waves, and hence acoustic energy, to leave the computational domain with minimal reflection; see, e.g.,
\cite{EngquistMajda1977,ShevchenkoKaltenbacher2015,
MuhrNikolicWohlmuth2019}.

Motivated by these considerations, in this work we study the Westervelt equation, one of the classical models for finite-amplitude sound propagation in nonlinear acoustics, supplemented with a first-order nonlinear absorbing boundary condition. More precisely, the acoustic velocity potential $u(t,{\bf x})$ is modeled by
\begin{subequations} \label{eq:westervelt}
\begin{align}
    \big(1 - 2\beta({\bf x})  \dot u\big)\,  \ddot u
    - \nabla \cdot \big(c^2({\bf x}) \nabla  u\big)
    - \nabla \cdot \big(\alpha({\bf x}) \nabla  \dot u\big)
    &= 0, \qquad\quad (t,{\bf x}) \in (0,T) \times \Omega, \label{eq:westervelt_1}\\
    c({\bf x})\sqrt{1 - 2\beta({\bf x}) \dot u} \dot u + (c^2({\bf x})\nabla u + \alpha({\bf x}) \nabla \dot u)\cdot{\bf n} &= 0,  \qquad \quad (t,{\bf x}) \in (0,T)\times\partial\Omega, \label{eq:westervelt_2}\\
    u(t=0, {\bf x}) =  u_0({\bf x}), \quad  \dot u(t=0, {\bf x}) &=  u_1({\bf x}), \quad\  {\bf x} \in \Omega,\label{eq:westervelt_3}
    \end{align}
\end{subequations}
where $\Omega \subset \mathbb{R}^d$, ${\bf n}$ denotes the unit outward normal of $\Omega$, and $(\dot{u}, \ddot{u})$ denote the first- and second-order time derivatives of $u$. Moreover, $c({\bf x}) > 0$ denotes the spatially varying sound speed, $\alpha({\bf x}) \geq 0$ is the diffusivity, and $\beta({\bf x}) \in \mathbb{R}$ characterizes the strength of the acoustic nonlinearity. 
We note that the Westervelt equation is also frequently formulated in terms of the acoustic pressure $p(t,{\bf x})$. In a homogeneous medium, these two formulations are closely related through $p=\varrho \dot u$, where $\varrho$ denotes the mass density. Differentiating the velocity-potential formulation in time and identifying the corresponding physical parameters leads to the pressure formulation of the Westervelt equation; see, for example,
\cite{HamiltonBlackstock1998,Nikolic2015}. A characteristic analytical difficulty of the Westervelt equation \eqref{eq:westervelt} is the solution-dependent coefficient $a(\dot u):=1-2\beta\dot u$, which multiplies the highest-order time derivative and may cause degeneracy if it approaches zero. Consequently, a uniform positive lower bound on $a(\dot u)$ plays a central role in both the continuous and discrete analyses.

To date, the Westervelt equation has been studied extensively from both analytical and numerical perspectives. On the analytical side, much of the classical well-posedness theory has been developed for the acoustic-pressure formulation. Under homogeneous Dirichlet boundary conditions, global well-posedness and exponential decay for sufficiently small data were established in \cite{KaltenbacherLasiecka2009}, while optimal regularity and long-time behavior were studied in \cite{MeyerWilke2011}; further extensions to more general coefficients, damping mechanisms, and vanishing sound diffusivity can be found in
\cite{BrunnhuberKaltenbacherRadu2014,KaltenbacherNikolic2022}. The treatment of absorbing boundaries has formed another important direction of research. For the acoustic-pressure formulation, nonlinear absorbing boundary conditions of different orders were derived in \cite{ShevchenkoKaltenbacher2015}, with corresponding well-posedness and energy analyses developed in \cite{SimonettWilke2017,KaltenbacherShevchenko2019}. For the acoustic-velocity-potential formulation considered in the present work, equipped with the nonlinear absorbing boundary condition in \eqref{eq:westervelt}, a corresponding well-posedness theory is developed in the companion work \cite{wang2026wellposed}. On the numerical side, rigorous finite element analyses have likewise been developed for several forms of the Westervelt equation. For the acoustic-pressure formulation, \cite{NikolicWohlmuth2019} established a priori error estimates for a conforming finite element discretization using a variable-coefficient linearization and a Banach fixed-point argument, while more general finite element frameworks and robust approximations with respect to small sound diffusivity were subsequently studied in \cite{HochbruckMaier2022,MelianiNikolic2024,Nikolic2025}. For the acoustic-velocity-potential formulation, \cite{AntoniettiEtAl2020} developed a high-order interior-penalty discontinuous Galerkin (DG) method, and \cite{GomezMeliani2025} proposed an asymptotic-preserving hybridizable DG method; both analyses were carried out under homogeneous Dirichlet boundary conditions. Nonlinear absorbing boundary conditions for the velocity-potential formulation were studied computationally in \cite{MuhrNikolicWohlmuth2019} using conforming finite elements. We also mention the recent spatial DG analysis \cite{deWitNikolic2026} for the acoustic-pressure formulation with linear absorbing boundary conditions in a coupled model for ultrasound-enhanced drug delivery.

The aim of this work is to develop an arbitrary order energy-based DG (EDG) method for the Westervelt equation \eqref{eq:westervelt}. The EDG method was originally introduced in \cite{AppeloHagstrom2015} for linear acoustic wave equations in second-order form. Its construction is guided directly by the energy balance of the underlying equation: a variable weakly equivalent to the momentum variable is introduced, and numerical fluxes are designed according to the energy flux across element boundaries. In contrast to DG formulations based on a complete first-order reduction in space, the EDG approach retains the second-order spatial structure and introduces only one auxiliary variable, independently of the spatial dimension. Moreover, it
admits simple mesh-independent numerical fluxes whose parameters directly determine whether the resulting scheme is energy conserving or dissipating; see \cite{AppeloHagstrom2018,ZhangHagstromAppelo2019,appelo2020energy,ren2024energy,zhang2024local,zhang2022energy,appelo2023energy,hagstrom2021discontinuous,du2019convergence,zhang2026energy} and the references therein.

In this work, we extend the EDG framework to the nonlinear Westervelt equation with spatially varying acoustic parameters and absorbing boundary conditions \eqref{eq:westervelt}. Introducing $v=\dot u$, we construct the semi-discrete formulation directly from the continuous energy balance and design numerical fluxes at both interelement and physical boundaries so that their dissipative properties can be characterized explicitly. A key feature of the proposed formulation compared with the existing EDG formulations is the treatment of the strong damping operator $-\nabla\cdot(\alpha\nabla v)$. To preserve the underlying energy-flux structure, we augment the standard EDG weak formulation with an additional consistent term associated with the strong-damping flux. Although this term vanishes for the exact solution, it completes the discrete interface and boundary pairing associated with the total spatial flux $c^2\nabla u_h+\alpha\nabla v_h$, where $u_h$ and $v_h$ denote the EDG approximations of $u$ and $v$, respectively, thereby allowing the wave and strong-damping contributions to be treated within a unified energy framework. In addition, the numerical flux at the absorbing boundary is constructed consistently with \eqref{eq:westervelt_2}, and therefore depends explicitly on the solution-dependent coefficient $a_h = 1-2\beta v_h$.

The main contributions of this work are fivefold. First, we formulate an EDG method for the acoustic-velocity-potential Westervelt equation with spatially varying material parameters and nonlinear absorbing boundary conditions, including a consistent treatment of the strong-damping flux. Second, we establish stability estimates for the variable-coefficient linearized EDG problem under a general family of dissipative numerical fluxes, with particular attention to the additional boundary terms induced by the time-dependent coefficient $a_h$. Third, we derive a priori error estimates using a projection adapted to the energy structure of the formulation. Fourth, we develop and analyze a compatible initialization procedure for the nonlinear EDG scheme, establishing its well-posedness and showing that the resulting discrete initial data satisfy the high-order accuracy requirements needed to achieve the desired convergence rates. Finally, we establish existence and uniqueness of the nonlinear semi-discrete solution and derive convergence estimates by means of a fixed-point argument under suitable regularity, smallness, and nondegeneracy assumptions.

The remainder of the paper is organized as follows. Section~\ref{sec:semidiscrete} introduces the continuous energy structure of the Westervelt equation \eqref{eq:westervelt} and presents the corresponding semi-discrete EDG formulation together with the numerical fluxes at interelement and physical boundaries. Section~\ref{sec:linearized_EDG_scheme} is devoted to the stability and error analysis of the associated variable-coefficient linearized problem. Section~\ref{sec:nonlinear} addresses the construction of compatible discrete initial data, and establishes existence, uniqueness, and convergence of the nonlinear semi-discrete solution using a fixed-point argument. Numerical experiments illustrating the theoretical results and the performance of the proposed method are presented in Section~\ref{sec:numerics}. Finally, the Appendix proves the lemma concerning the construction and accuracy of the discrete initial data.

\section{Semi-discrete EDG formulation}\label{sec:semidiscrete}

To establish the EDG scheme for \eqref{eq:westervelt}, we introduce a new variable $v = \dot u$, which leads to the first-order system in time
\begin{subequations}\label{eq:1st_in_time_sys}
\begin{align}
    \dot u - v &= 0,  \label{eq:1st_in_time_sys_1}\\
    \big(1-2\beta({\bf x})v\big) \dot v - \nabla \cdot \big(c^2({\bf x}) \nabla u\big) - \nabla \cdot \big(\alpha({\bf x}) \nabla v\big) &= 0. \label{eq:1st_in_time_sys_2}
\end{align}
\end{subequations}
Given the nondegeneracy condition
\begin{equation}\label{eq:non_dege_conti}
1-2\beta({\bf x}) v>0,
\end{equation}
the system \eqref{eq:1st_in_time_sys} admits the nonnegative energy functional
\begin{equation}\label{eq:energy_conti}
    E(t) := \int_{\Omega} \frac{1}{2} c^2({\bf x}) |\nabla u|^2 + \Big(\frac{1}{2} - \frac{2}{3} \beta({\bf x}) v\Big) v^2 \, d{\bf x} \geq 0.
\end{equation}
Moreover, under the boundary condition \eqref{eq:westervelt_2}, we find that the energy \eqref{eq:energy_conti} is dissipated over time:
\[\begin{aligned}
\dot E(t) &= \int_\Omega c^2({\bf x})\nabla u\cdot \nabla \dot u + \big(1 - 2\beta({\bf x})v\big)v\dot v \,d{\bf x} \\
       & = -\int_\Omega \alpha |\nabla v|^2 \,d{\bf x} - \int_{\pp \Omega} c \sqrt{1-2\beta v} |v|^2 \, dS \leq 0.
\end{aligned}\]
Thus, we have $E(t) \leq E(0)$ for all $t\in (0,T)$.
In the following, we introduce the EDG spatial discretization of \eqref{eq:westervelt} based on the reformulation \eqref{eq:1st_in_time_sys}, and establish its fundamental energy estimate.

\subsection{EDG spatial discretization}  
Let $\mathcal T_h$ be a tessellation of the computational domain $\Omega$, consisting of nonoverlapping elements $\Omega_j$ such that $\bar{\Omega}=\cup_{\Omega_j\in\mathcal T_h}\bar \Omega_j$. Throughout the theoretical analysis, we assume that $\{\mathcal{T}_h\}_h$ is a shape-regular and quasi-uniform family of affine simplicial or tensor-product meshes. These assumptions ensure that the Bramble--Hilbert approximation estimates, as well as the inverse and trace inequalities used in the analysis below, hold with constants independent of the mesh size $h:=\max_{\Omega_j\in\mathcal T_h} h_j$, where $h_j$ denotes the diameter of the element $\Omega_j$.

To approximate the solution components $(u,v)$ on each element $\Omega_j$, we use either polynomials or the tensor product of polynomials of degree $q_u$ and $q_v$, respectively\footnote{For simplicity we abuse notation and let $\mathcal{P}^r, r = q_u, q_v,$ denote either the polynomials of degree $r$ or the tensor-product polynomials of degree $r$ in each coordinate on a reference element.}. Specifically, we define the sets $U_h^{q_u}$ and $V_h^{q_v}$ as
\begin{equation}\label{eq:approximation_space}
\begin{aligned}
U_h^{q_u} &:= \{w_h({\bf x})\big| w_h({\bf x}) \in \mathcal{P}^{q_u}(\Omega_j), {\bf x} \in \Omega_j\quad \forall \Omega_j \in \mathcal{T}_h \},\\
V_h^{q_v} &:= \{w_h({\bf x})\big| w_h({\bf x}) \in \mathcal{P}^{q_v}(\Omega_j), {\bf x} \in \Omega_j\quad \forall \Omega_j \in \mathcal{T}_h \},
\end{aligned}
\end{equation}
and look for an approximation to the system \eqref{eq:1st_in_time_sys} that satisfies a discrete energy estimate, analogous to \eqref{eq:energy_conti}. To this end, we consider the following discrete energy on each element $\Omega_j$
\begin{equation}\label{eq:energy_discrete}
    E_j^h(t) := \int_{\Omega_j} \frac{1}{2} c^2({\bf x}) |\nabla u_h|^2 + \Big(\frac{1}{2} - \frac{2}{3} \beta({\bf x}) v_h\Big) v_h^2 \, d{\bf x},
\end{equation}
and its time derivative 
\begin{equation}\label{eq:energy_discrete_dt}
    \dot E_j^h(t) = \int_{\Omega_j} c^2({\bf x}) \nabla u_h \cdot \nabla \dot u_h + \big(1-2\beta({\bf x}) v_h\big) v_h \dot v_h \, d{\bf x}.
\end{equation}
Here, for the moment, we assume that $1-2\beta v_h > 0$, the discrete counterpart of the nondegeneracy condition \eqref{eq:non_dege_conti}; this assumption is verified in Section~\ref{sec:nonlinear}.

To develop a discrete weak form compatible with the discrete energy \eqref{eq:energy_discrete}--\eqref{eq:energy_discrete_dt}, we choose $(\phi_u,\phi_v) \in U_h^{q_u}\times V_h^{q_v}$, test \eqref{eq:1st_in_time_sys_1} with $-\nabla\cdot(c^2\nabla \phi_u)$ and \eqref{eq:1st_in_time_sys_2} with $\phi_v$, and add boundary integrals that vanish in the continuous problem to obtain
\[\begin{aligned}
&\int_{\Omega_j} -\nabla \cdot (c^2 \nabla \phi_u) \big(\dot u_h - v_h\big) \,d{\bf x} = \int_{\partial\Omega_j} c^2\nabla \phi_u \cdot{\bf n}\big(v_h^\ast - {\dot u_h} \big)\,dS,  \\
  & \int_{\Omega_j} \phi_v\Big((1-2\beta v_h) \dot v_h - \nabla \cdot (c^2 \nabla u_h) - \nabla \cdot (\alpha \nabla v_h) \Big)\,d{\bf x} \\
  = & \int_{\partial\Omega_j} \phi_v \big((c^2\nabla u_h)^\ast - c^2\nabla u_h\big)\cdot{\bf n} + \phi_v \big((\alpha \nabla v_h)^\ast - \alpha \nabla v_h\big)\cdot{\bf n} + \alpha \nabla\phi_v\cdot{\bf n}\big(v_h^\ast - v_h\big)\,dS.
\end{aligned}\]
Here, ${\bf n}$ denotes the unit outward normal of $\Omega_j$; $v_h^\ast, (c^2\nabla u_h)^\ast$ and $(\alpha \nabla v_h)^\ast$ are numerical fluxes at element boundaries. We have omitted ${\bf x}$ dependency in $c^2, \beta$ and $\alpha$ for simplicity. Applying integration by parts then yields the following alternative formulation,
\begin{subequations}\label{eq:EDG}
\begin{align}
&\int_{\Omega_j} c^2\nabla \phi_u \cdot \nabla(\dot u_h - v_h) \,d {\bf x}= \int_{\pp \Omega_j} c^2 \nabla   \phi_u \cdot \vec{n} \big(v_h^* - v_h\big) \, dS, \label{eq:EDG_1}\\
    &\int_{\Omega_j} \phi_v(1-2\beta v_h)\dot v_h+ c^2\nabla \phi_v \cdot \nabla u_h + \alpha \nabla \phi_v \cdot \nabla v_h  \, d{\bf x}\nonumber\\
    = & \int_{\pp \Omega_j}  \phi_v(\alpha\nabla v_h)^\ast \cdot{\bf n} +  \phi_v (c^2\nabla u_h)^* \cdot \vec{n} + \alpha \nabla \phi_v \cdot{\bf n}  (v_h^\ast - v_h) \,dS. \label{eq:EDG_2} 
\end{align}
\end{subequations}
Finally, to uniquely determine $\dot u_h$, we supplement \eqref{eq:EDG_1} with an additional equation that determines the mean value of $\dot u_h$ over $\Omega_j$ by
\begin{equation}\label{eq:mean_value}
    \int_{\Omega_j} \tilde{\phi}(\dot u_h - v_h)\,d{\bf x} = 0 \qquad \forall \tilde {\phi} \in \mathcal{P}^0(\Omega_j).
\end{equation}
We note that \eqref{eq:mean_value} is consistent with \eqref{eq:1st_in_time_sys_1}; moreover, it does not affect the discrete energy and is also sufficient to ensure that \eqref{eq:EDG_1} holds when $\phi_u \in \mathcal{P}^0(\Omega_j)$.

\subsection{Numerical fluxes}
To finalize the EDG scheme \eqref{eq:EDG}--\eqref{eq:mean_value}, we must prescribe the states $v_h^*$, $(c^2\nabla u_h)^*$, and $(\alpha\nabla v_h)^*$ at both interelement and physical boundaries. Label two elements sharing an interelement boundary face, denoted by $F_j$, by $\Omega_j^1$ and $\Omega_j^2$. Moreover, let $w_h$ be a continuously differentiable scalar function on $\Omega_j^1$ and $\Omega_j^2$, and $w_{h,1} := (w_h|_{\Omega_j^1})|_{F_j}$, $w_{h,2} := (w_h|_{\Omega_j^2})|_{F_j}$. We then introduce standard notations for weighted averages and jumps given by
\begin{equation*}
    \dbrace{w_h}_\theta = \theta w_{h,1} + (1-\theta) w_{h,2}, \quad \dbrack{w_h} = w_{h,1} \vec{n}_1 + w_{h,2} \vec{n}_2,
\end{equation*}
and 
\begin{equation*}
    \dbrace{\nabla w_h}_{1-\theta} = (1-\theta) \nabla w_{h,1} + \theta \nabla w_{h,2}, \quad \dbrack{\nabla w_h} = \nabla w_{h,1} \cdot \vec{n}_1 + \nabla w_{h,2} \cdot \vec{n}_2.
\end{equation*}
Here, ${\bf n}_1$ and ${\bf n}_2$ are the unit outward normals to $\partial \Omega_j^1$ and $\partial \Omega_j^2$, respectively.

\textbf{Interelement boundaries.} First, for the interelement boundaries, we introduce the following numerical fluxes at an interelement boundary face $F_j$,
\begin{equation}\label{eq:generalfl}
\begin{aligned}
    v_h^* & := \dbrace{v_h}_{\theta} - \tau \big( \dbrack{c^2\nabla u_h} +  \dbrack{\alpha\nabla v_h} \big),\\
    (c^2\nabla u_h)^* & := \dbrace{c^2\nabla u_h}_{1-\theta} - \lambda_u \dbrack{v_h},\\
    (\alpha\nabla v_h)^* & := \dbrace{\alpha \nabla v_h}_{1-\theta} - \lambda_v \dbrack{v_h},
\end{aligned}
\end{equation}
where $\theta \in \mathbb{R}$, and $\tau, \lambda_u, \lambda_v \geq 0$, and find that the net contribution to the discrete energy from the interelement faces is nonpositive:
\begin{equation}\label{eq:inn}
    \begin{aligned}
   & \sum_j \int_{F_j} \big(c_1^2\nabla u_{h,1}  + \alpha_1\nabla v_{h,1}\big)\cdot \vec{n}_1(v_h^* - v_{h,1}) + v_{h,1} \big((c^2\nabla u_h)^*  + (\alpha\nabla v_h)^* \big)\!\cdot\! \vec{n}_1 \\
    &\hspace{1cm} + \big(c_2^2\nabla u_{h,2} + \alpha_2\nabla v_{h,2}\big)\cdot \vec{n}_2(v_h^* - v_{h,2}) + v_{h,2} \big((c^2\nabla u_h)^*  + (\alpha\nabla v_h)^* \big)\!\cdot\! \vec{n}_2\,dS \\
    = &  -\, \underbrace{\sum_j\int_{F_j}(\lambda_u + \lambda_v) |\dbrack{v_h}|^2 + \tau \big( \dbrack{c^2\nabla u_h} + \dbrack{\alpha \nabla v_h}\big)^2\,dS}_{\mathcal D_{I}(u_h,v_h)} \leq 0.
\end{aligned}
\end{equation}
Note that the commonly used \textit{central fluxes} can be obtained by setting $\theta = 1/2$ and $\tau =\lambda_u= \lambda_v = 0$ in \eqref{eq:generalfl}, that is,
\begin{equation*}
    v_h^* = \dbrace{v_h}_{\frac{1}{2}}, \quad (c^2\nabla u_h)^* = \dbrace{c^2\nabla u_h}_{\frac{1}{2}}, \quad (\alpha \nabla v_h)^* = \dbrace{\alpha \nabla v_h}_{\frac{1}{2}}.   
\end{equation*}
On the other hand, if $\tau = \lambda_u = \lambda_v = 0$, and $\theta = 1$ or $\theta=0$, the \textit{alternating fluxes} are obtained, respectively, as
\begin{equation*}
    v_h^* = v_{h,1}, \quad (c^2\nabla u_h)^* = c_2^2\nabla u_{h,2}, \quad (\alpha \nabla v_h)^* = \alpha_2\nabla v_{h,2},
\end{equation*}
or 
\begin{equation*}
    v_h^* = v_{h,2}, \quad (c^2\nabla u_h)^* = c_1^2\nabla u_{h,1}, \quad (\alpha \nabla v_h)^* = \alpha_1\nabla v_{h,1}.   
\end{equation*}
We note that both central fluxes and alternating fluxes yield $ \mathcal{D}_{I}(u_h,v_h) = 0$.

\textbf{Physical boundaries.} Next, for the physical boundaries, define the boundary residual due to the numerical approximation of the boundary condition \eqref{eq:westervelt_2} by
\begin{equation}\label{eq:boundary_residual}
\rho(u_h,v_h) := (c^2 \nabla u_h  + \alpha \nabla v_h) \cdot \vec{n} +  c\sqrt{1-2\beta v_h} v_h.
\end{equation}
Let $\sigma\in[-1,1]$ and let $B_j$ denote the physical boundary faces. The one-parameter family of numerical fluxes
\begin{equation}\label{eq:flux_bc}
\begin{aligned}
v_h^* &= v_h -\frac{1-\sigma}{2c\sqrt{1-2\beta v_h}}\rho(u_h,v_h), \\
    {\bf w}_h^* &= {\bf w}_h - \frac{1+\sigma}{4}\rho(u_h,v_h)\mathbf n, \qquad {\bf w}_h \in \{c^2\nabla u_h, \alpha \nabla v_h\}
\end{aligned}
\end{equation}
yields a nonpositive net contribution to the discrete energy:
\begin{equation} \label{eq:bccontribution}
\begin{aligned}
   &\sum_j\int_{B_j} \!\!\big(c^2\nabla u_{h} \cdot \vec{n} + \alpha\nabla v_{h}\cdot {\bf n}\big)(v_h^* - v_{h}) + v_{h} \big((c^2\nabla u_h)^* \!\cdot\! \vec{n} + (\alpha\nabla v_h)^* \!\cdot\! \vec{n}\big)\,dS \\
    = & - \sum_j \int_{B_j} c\sqrt{1-2\beta v_h}\, |v_h^*|^2 + \frac{1-\sigma^2}{4c\sqrt{1-2\beta v_h}}\rho^2(u_h,v_h)\,dS \leq 0.
    \end{aligned}
\end{equation}

To establish the well-posedness and convergence of the proposed EDG scheme \eqref{eq:EDG}--\eqref{eq:mean_value}, we proceed with the Banach fixed-point theorem \cite{banach1922operations} by first studying the well-posedness and convergence of a class of linear semi-discrete EDG schemes.

\section{Analysis of a class of linear semi-discrete schemes} \label{sec:linearized_EDG_scheme}
Consider the following linear problem
\begin{equation}
\begin{aligned}
    a(t, {\bf x})\, \ddot u
    - \nabla \cdot \big(c^2({\bf x}) \nabla  u\big)
    - \nabla \cdot \big(\alpha({\bf x}) \nabla  \dot u\big)
    &= 0, \qquad\quad (t,{\bf x}) \in (0,T) \times \Omega, \\
    c\sqrt{a(t,{\bf x})} \dot u + (c^2\nabla u + \alpha \nabla \dot u)\cdot{\bf n} &= 0,  \qquad \quad (t,{\bf x}) \in (0,T)\times\partial\Omega, \\
    u(t=0, {\bf x}) =  u_0({\bf x}), \quad  \dot u(t=0, {\bf x}) &=  u_1({\bf x}), \quad\  {\bf x} \in \Omega,
    \end{aligned}
    \label{eq:westervelt_linear}
\end{equation}
where $c({\bf x})$ and $\alpha({\bf x})$ are as in \eqref{eq:westervelt}. Assume that there exist constants $a_0,a_1>0$ such that
\begin{equation}\label{eq:a_condition}
a_0 \leq a(t,{\bf x}) \leq a_1 \quad \forall (t,{\bf x}) \in [0,T] \times \bar\Omega.
\end{equation}
Sufficient conditions for the well-posedness of \eqref{eq:westervelt_linear} can be found in \cite{wang2026wellposed}.

Similar to the nonlinear scheme \eqref{eq:EDG}--\eqref{eq:mean_value}, we introduce the EDG semi-discretization of \eqref{eq:westervelt_linear} by seeking $(u_h,v_h)\in U_h^{q_u}\times V_h^{q_v}$ such that, for every $(\phi_u,\phi_v,\tilde\phi)\in U_h^{q_u}\times V_h^{q_v}\times \mathcal P^0(\Omega_j)$,
\begin{subequations}\label{eq:linearized_scheme}
\begin{align}
\int_{\Omega_j} c^2\nabla\phi_u \cdot \nabla \dot u_h\, d{\bf x} + \mathcal K_j(u_h,v_h;\phi_u;a_h) &= 0,\label{eq:linearized_scheme_1}\\
\int_{\Omega_j} \phi_va_h\dot v_h \,d{\bf x} + \mathcal{H}_j(u_h,v_h;\phi_v;a_h) &= 0,\label{eq:linearized_scheme_2}\\
\int_{\Omega_j} \tilde {\phi}(\dot u_h - v_h)\,d{\bf x} &= 0,\label{eq:linearized_scheme_3}
\end{align}
\end{subequations}
where $a_h := a_h(t,{\bf x})$ denotes a prescribed coefficient in the semi-discrete problem, and
\begin{equation}\label{eq:def_KH}
\begin{aligned}
\mathcal K_j(u_h,v_h;\phi_u;a_h) =&  -\int_{\Omega_j} c^2\nabla \phi_u \cdot \nabla v_h\,d{\bf x} - \int_{\pp \Omega_j} c^2 \nabla   \phi_u \cdot \vec{n} \big(v_h^* - v_h\big) \, dS,\\
\mathcal H_j(u_h,v_h;\phi_v;a_h) = &  \int_{\Omega_j} c^2\nabla \phi_v \cdot \nabla u_h+ \alpha \nabla \phi_v \cdot \nabla v_h  \, d{\bf x}  \\
&- \int_{\pp \Omega_j}  \phi_v(\alpha\nabla v_h)^\ast \cdot{\bf n} +  \phi_v (c^2\nabla u_h)^* \cdot \vec{n} + \alpha \nabla \phi_v \cdot{\bf n}  (v_h^\ast - v_h) \,dS.
\end{aligned}
\end{equation}
The dependence of $\mathcal K_j$ and $\mathcal H_j$ on $a_h$ arises from the definition of the numerical fluxes. More precisely, on the interelement boundaries, the numerical fluxes $v_h^\ast$, $(c^2\nabla u_h)^\ast$, and $(\alpha\nabla v_h)^\ast$ are defined exactly as in \eqref{eq:generalfl}. On the physical boundaries, they are defined analogously to those in \eqref{eq:flux_bc}, namely,
\begin{subequations}\label{eq:flux_bc_linear}
\begin{align}
v_h^\ast &= v_h - \frac{1-\sigma}{2c\sqrt{a_h}}\rho^{\rm lin}(u_h,v_h; a_h),\label{eq:flux_bc_linear_1}\\
 {\bf w}_h^\ast &= {\bf w}_h - \frac{1+\sigma}{4}\rho^{\rm lin}(u_h,v_h; a_h) {\bf n},\qquad {\bf w}_h \in \{c^2\nabla u_h, \alpha\nabla v_h\},\label{eq:flux_bc_linear_2}
\end{align}
\end{subequations}
where
\begin{equation*}\label{eq:rho_def}
    \rho^{\rm lin}(u_h,v_h; a_h) = (c^2\nabla u_h) \cdot {\bf n} + (\alpha \nabla v_h) \cdot \mathbf{n} + c\sqrt{a_h}\,v_h.
\end{equation*}

\begin{remark}
For the stability analysis, we introduce the admissible coefficient class
\begin{equation}\label{eq:admissible_set}
\mathcal A_{\rm adm} : = \left\{b_h: \begin{aligned}
&0<a_\ast \le b_h(t,{\bf x})\leq a^\ast<\infty \quad \forall (t,{\bf x})\in[0,T]\times\bar\Omega; \\
&\|\dot b_h/b_h\|_{{L^1}(0,T;L^\infty(\Omega))} \leq M_a;\quad  \|\dot b_h/b_h\|_{{L^\infty}(0,T;L^\infty(\partial\Omega))} \leq \mu
\end{aligned}\right\},
\end{equation}
where $0<a_\ast < 1 <a^\ast<\infty$, $M_a\in(0,1)$, and $\mu \in (0,4)$ are fixed constants independent of $h$. Throughout the linear stability analysis, we assume that $a_h \in \mathcal A_{\rm adm}$. No approximation property relating $a_h$ to the exact coefficient $a$ is required for the stability result. Such a relation is needed only in the error analysis, where the discrepancy between $a_h$ and $a$, together with that between their time derivatives, is measured explicitly by the quantity $\mathcal Q(t)$, defined below in \eqref{eq:Q_def}.
\end{remark}

\textbf{Standard notation.} Before proceeding with the analysis of the linear scheme \eqref{eq:linearized_scheme}, we introduce some standard notation to simplify the presentation: (i) Throughout the paper, $C>0$ denotes a generic positive constant that is independent of the mesh size $h$, although its value may vary from line to line. (ii) For any integer $m \ge 0$ and $1 \le p \le \infty$, let $W^{m,p}(\Omega_j)$ denote the standard Sobolev space on the element $\Omega_j$, equipped with norm $\|\cdot\|_{m,p,\Omega_j}$ and seminorm $|\cdot|_{m,p,\Omega_j}$. When $p=2$, we write $H^m(\Omega_j)=W^{m,2}(\Omega_j)$ with norm $\|\cdot\|_{H^m(\Omega_j)}$ and seminorm $|\cdot|_{H^m(\Omega_j)}$. When $m=0$, $L^p(\Omega_j)$ denotes the usual Lebesgue space with norm $\|\cdot\|_{L^p(\Omega_j)}$. We further define the global norms and seminorms by {$\|v\|_{L^p(\Omega)}^p := \sum_j \|v\|_{L^p(\Omega_j)}^p$ for $1\le p<\infty$, $\|v\|_{L^\infty(\Omega)} := \max_j \|v\|_{L^\infty(\Omega_j)}$,} and $|v|_{H^m(\Omega)}^2 := \sum_j |v|_{H^m(\Omega_j)}^2$. (iii) For piecewise-defined functions, integrals over $\Omega$ are understood elementwise, namely, $\int_{\Omega} f\,d{\bf x} := \sum_j \int_{\Omega_j} f\,d{\bf x}$. When needed, $\nabla_h$ denotes the elementwise gradient, defined by $(\nabla_h v)|_{\Omega_j} := \nabla(v|_{\Omega_j})$. Similarly, boundary integrals over the physical boundary $\partial\Omega$ are understood as sums over the corresponding element boundaries, that is, $\int_{\partial\Omega} f\,d{S} := \sum_j \int_{B_j} f\,dS$, where $B_j$ represents physical boundaries.

\subsection{Well-posedness of the linear semi-discrete scheme}
We first establish the well-posedness of the semi-discrete problem \eqref{eq:linearized_scheme}. To this end, we introduce the discrete energy over each element $\Omega_j$ by
\begin{equation}\label{eq:linear_discrete_energy}
E_{h,j}^{\rm lin}(t) := \frac{1}{2} \int_{\Omega_j} a_h \big(|v_h|^2 + |\dot v_h|^2\big) +  c^2 \big(|\nabla u_h|^2 + |\nabla \dot u_h|^2\big)\,d{\bf x}.
\end{equation}
We note that, for the subsequent analysis of the nonlinear problem, it is convenient to incorporate higher-order time-derivative terms into the discrete energy. This stronger energy framework yields estimates that are compatible with the later error analysis.

\begin{truth}\label{thm1}
Let $c({\bf x}) > 0, \alpha({\bf x}) \geq 0$, and let the dissipation parameters in the numerical fluxes \eqref{eq:generalfl} and \eqref{eq:flux_bc_linear} satisfy $\lambda_u,\lambda_v,\tau \geq 0$ and $\sigma\in[-1,1]$. Assume that $a_h \in \mathcal A_{\rm adm}$, where $\mathcal A_{\rm adm}$ is defined in \eqref{eq:admissible_set}. Then, the linear semi-discrete problem \eqref{eq:linearized_scheme} admits a unique solution $(u_h,v_h)\in U_h^{q_u}\times V_h^{q_v}$ for any given initial data $(u_h(0),v_h(0))$, where $(U_h^{q_u}, V_h^{q_v})$ is defined in \eqref{eq:approximation_space}. Moreover, with $E_{h,j}^{\rm lin}(t)$ defined in \eqref{eq:linear_discrete_energy} and
\begin{equation}\label{eq:discreteenergy_linear}
E_h^{\rm lin}(t) := \sum_j E_{h,j}^{\rm lin}(t),
\end{equation}
there holds $E_h^{\rm lin}(t) \leq   \frac{1}{1-M_a} E_h^{\rm lin}(0)$ for $t\in[0,T]$.
\end{truth}

\begin{proof}
Since the discrete spaces $(U_h^{q_u}, V_h^{q_v})$ are finite dimensional, the semi-discrete problem \eqref{eq:linearized_scheme} can be written as a finite-dimensional linear ODE system with time-dependent coefficients. Because $a_h\ge a_\ast>0$, the corresponding mass matrix is uniformly invertible. Standard ODE theory therefore yields a unique local solution on some interval $[0,T_h]$, where $0<T_h\le T$. We next establish a stability estimate, which allows us to extend the solution to the whole interval $[0,T]$.

Choosing $(\phi_u, \phi_v) = (u_h, v_h)$ in \eqref{eq:linearized_scheme}, and summing the resulting equations over all elements, we obtain
\begin{equation}\label{eq:aux_linear1}
\int_\Omega c^2\nabla_h u_h\cdot\nabla_h {\dot u_h} + a_hv_h\dot v_h \,d{\bf x} + \mathcal K(u_h,v_h;u_h;a_h) + \mathcal H(u_h,v_h;v_h;a_h) = 0,
\end{equation}
where
\begin{equation}\label{eq:defKH_total}
\mathcal K(u_h,v_h;u_h;a_h) = \sum_j\mathcal{K}_j(u_h,v_h;u_h;a_h),\qquad  \mathcal H(u_h,v_h;v_h;a_h) = \sum_j \mathcal H_j(u_h,v_h;v_h;a_h).
\end{equation}
Invoking the definitions of $\mathcal K_j$ and $\mathcal H_j$ in \eqref{eq:def_KH}, the definitions of the numerical fluxes on the interelement boundaries in \eqref{eq:generalfl} and on the physical boundary in \eqref{eq:flux_bc_linear}, and proceeding as in the derivations of \eqref{eq:inn} and \eqref{eq:bccontribution}, we obtain
\begin{equation}\label{eq:KH}
\mathcal K(u_h,v_h;u_h;a_h) + \mathcal H(u_h,v_h;v_h;a_h)
= \int_{\Omega} \alpha |\nabla_h v_h|^2\,d{\bf x} +\mathcal D_{I}(u_h,v_h) + \mathcal D_B(u_h,v_h;a_h),
\end{equation}
where $\mathcal D_I$ is defined in \eqref{eq:inn}, and 
\begin{equation}\label{eq:def_DB}
\mathcal D_B(u_h,v_h;a_h) = \int_{\partial\Omega} \frac{(1+\sigma)c\sqrt{a_h}}{2}|v_h|^2 + \frac{1-\sigma}{2c\sqrt{a_h}}|(c^2\nabla_h u_h + \alpha \nabla_h v_h)\cdot{\bf n}|^2\, dS.
\end{equation}

To get higher-order terms in the discrete energy, we take the time derivative of \eqref{eq:linearized_scheme} to obtain
\begin{subequations}\label{eq:linearized_scheme_dt}
\begin{align}
\int_{\Omega_j} c^2\nabla\phi_u \cdot \nabla \ddot u_h\, d{\bf x} + \dot{\mathcal K}_j(u_h,v_h;\phi_u;a_h) &= 0,\label{eq:linearized_scheme_dt_1}\\
\int_{\Omega_j} \phi_v \dot a_h \dot v_h + \phi_v a_h\ddot v_h \,d{\bf x} + \dot{\mathcal{H}}_j(u_h,v_h;\phi_v;a_h) &= 0.\label{eq:linearized_scheme_dt_2}
\end{align}
\end{subequations}
Choosing $(\phi_u, \phi_v) = (\dot u_h, \dot v_h)$ in \eqref{eq:linearized_scheme_dt} and summing the resulting equations over all elements yields
\begin{equation}\label{eq:aux_linear2}
\int_\Omega  c^2 \nabla_h \dot u_h\cdot \nabla_h \ddot u_h +  \dot a_h|\dot v_h|^2 + a_h \dot v_h \ddot v_h \,d{\bf x} + \dot{\mathcal K}(u_h,v_h;\phi_u;a_h)\big|_{\phi_u = \dot u_h} + \dot{\mathcal H}(u_h,v_h;\phi_v;a_h)\big|_{\phi_v = \dot v_h} = 0.
\end{equation}
Analogously to the derivation of \eqref{eq:KH}, we obtain
\begin{equation}\label{eq:dot_KH}
\begin{aligned}
\dot{\mathcal K}(u_h&,v_h;\phi_u;a_h)\big|_{\phi_u = \dot u_h} + \dot{\mathcal H}(u_h,v_h;\phi_v;a_h)\big|_{\phi_v = \dot v_h}  = \int_\Omega \!\!\alpha|\nabla_h \dot v_h|^2\,d{\bf x} +  \mathcal D_{I}(\dot u_h,\dot v_h) + \mathcal D_B(\dot u_h,\dot v_h;a_h) \\
&+\underbrace{\int_{\partial\Omega} \frac{(1+\sigma)c\sqrt{a_h}}{4}\frac{\dot a_h}{a_h}v_h\dot v_h
- \frac{1-\sigma}{4c\sqrt{a_h}}\frac{\dot a_h}{a_h}\big(c^2\nabla u_h+\alpha\nabla v_h\big)\cdot{\bf n}\big(c^2\nabla \dot u_h+\alpha\nabla \dot v_h\big)\cdot{\bf n}\,dS}_{\mathfrak E_B(u_h, v_h; a_h)},
\end{aligned}
\end{equation}
where $\mathcal D_I$ and $\mathcal D_B$ are defined in \eqref{eq:inn} and \eqref{eq:def_DB}, respectively; while the additional terms compared with \eqref{eq:KH} arise from differentiating the numerical fluxes \eqref{eq:flux_bc_linear} on the physical boundary in time.

Summing \eqref{eq:aux_linear1} and \eqref{eq:aux_linear2}, substituting \eqref{eq:KH} and \eqref{eq:dot_KH} into the resulting equation, and invoking the definition of the discrete energy in \eqref{eq:discreteenergy_linear}, we obtain
\begin{equation}\label{eq:aux1_linear_energy}
\begin{aligned}
&\dot E_h^{\rm lin}(t) + \int_{\Omega}\alpha \big(|\nabla_h v_h|^2  + |\nabla_h \dot v_h|^2\big)\,d{\bf x} + \mathcal{D}_I(u_h,v_h)+ \mathcal D_I(\dot u_h,\dot v_h) \\
= & \int_{\Omega} \frac{1}{2}\dot a_h|v_h|^2 - \frac{1}{2}\dot a_h |\dot v_h|^2\, d{\bf x} - \big(\mathcal D_B(u_h,v_h;a_h) + \mathcal D_B(\dot u_h, \dot v_h;a_h) + \mathfrak E_B(u_h,v_h;a_h) \big). 
\end{aligned}
\end{equation}
By the definition of  $\mathcal D_I$ in \eqref{eq:inn}, it is straightforward to see that 
\begin{equation}\label{eq:boundDI}
\mathfrak D_I(u_h,v_h) := \mathcal{D}_I(u_h,v_h)+ \mathcal D_I(\dot u_h,\dot v_h) \geq 0.
\end{equation}
To estimate the right-hand side of \eqref{eq:aux1_linear_energy}, we compute
\begin{equation}\label{eq:boundDB}
\begin{aligned}
& \overbrace{\mathcal D_B(u_h,v_h;a_h) + \mathcal D_B(\dot u_h, \dot v_h;a_h) + \mathfrak E_B(u_h,v_h;a_h)}^{\mathfrak D_B(u_h,v_h;a_h)}\\
=&\int_{\partial\Omega} \frac{(1+\sigma)c\sqrt{a_h}}{2}\Big(|v_h|^2 + |\dot v_h|^2 + \frac{1}{2}\frac{\dot a_h}{a_h}v_h\dot v_h\Big)
+\frac{1-\sigma}{2c\sqrt{a_h}}\Big(w_h^2 + \dot w_h^2
 - \frac{1}{2}\frac{\dot a_h}{a_h}w_h\dot w_h\Big)\,dS \\
 \geq & \ \Big(1-\frac{\mu}{4}\Big)\int_{\partial\Omega} \!\!\frac{(1+\sigma)c\sqrt{a_h}}{2}\big(|v_h|^2 + |\dot v_h|^2\big) + \frac{1-\sigma}{2c\sqrt{a_h}}\big(w_h^2 + \dot w_h^2\big)\,dS, \ \text{where } w_h\!=\!(c^2\nabla u_h\!+\!\alpha\nabla v_h)\!\cdot\!{\bf n}.
\end{aligned}
\end{equation}
Here, we have applied Young's inequality and the fact that $a_h\in \mathcal A_{\rm adm}$ to get the inequality. Moreover, since $\mu \in (0,4)$ and $\sigma\in[-1,1]$, it follows that $\mathfrak D_B(u_h,v_h;a_h) \geq 0$. Substituting \eqref{eq:boundDI} and \eqref{eq:boundDB} into \eqref{eq:aux1_linear_energy} yields
\begin{equation*}
\begin{aligned}
& \dot E_h^{\rm lin}(t) +  \int_{\Omega} \alpha (|\nabla_h \dot v_h|^2 +  |\nabla_h v_h|^2)\,d{\bf x} + \mathfrak{D}_I(u_h,v_h)\\
& +\Big(1-\frac{\mu}{4}\Big)\int_{\partial\Omega} \frac{(1+\sigma)c\sqrt{a_h}}{2}\big(|v_h|^2 + |\dot v_h|^2\big) + \frac{1-\sigma}{2c\sqrt{a_h}}\big(w_h^2 + \dot w_h^2\big)\,dS \\
  \leq  & \  \frac{1}{2}\Big\|\frac{\dot a_h}{a_h}\Big\|_{L^\infty(\Omega)}  \int_{\Omega} a_h (|\dot v_h|^2 + |v_h|^2)\,d{\bf x}, \qquad \text{where } w_h = \big(c^2\nabla u_h+\alpha\nabla v_h\big)\cdot{\bf n}.
  \end{aligned}
\end{equation*}
Integrating the above inequality over $[0,t]$ with $t\leq T_h$ and using the definition of $E_h^{\rm lin}(t)$ gives
\[
E_h^{\rm lin}(t)
\leq E_h^{\rm lin}(0) + \int_0^{T_h} \Big\|\frac{\dot a_h(s)}{a_h(s)}\Big\|_{L^\infty(\Omega)}\,ds\, \max_{s\in[0,T_h]}E_h^{\rm lin}(s).
\]
Consequently, given that $a_h\in \mathcal A_{\rm adm}$, we obtain
\begin{equation}\label{eq:stability_aux1}
\begin{aligned}
&\max_{t \in [0,T_h]} E_h^{\rm lin}(t)
\leq  \frac{1}{1-M_a}E_h^{\rm lin}(0).
\end{aligned}
\end{equation}

The bound in \eqref{eq:stability_aux1} is uniform with respect to $T_h$. Hence, the discrete energy cannot blow up, and the standard continuation argument extends the local solution to the entire interval $[0,T]$. This establishes existence on $[0,T]$. Uniqueness follows by applying the same energy estimate to the difference of two solutions with identical initial data.
\end{proof}

\begin{remark}
The boundary control in the admissible set $\mathcal A_{\rm adm}$ defined in \eqref{eq:admissible_set} is introduced solely to control the additional physical boundary terms arising from the time differentiation of the numerical fluxes \eqref{eq:flux_bc_linear} associated with the absorbing boundary condition; see \eqref{eq:boundDB}. Such control is required to derive the higher-order energy estimate and, consequently, the error estimate presented below. If, instead, Dirichlet boundary conditions are imposed, as considered in \cite{AntoniettiEtAl2020,GomezMeliani2025}, this additional condition is no longer necessary.
\end{remark}

\begin{remark}\label{remark2}
For the physical boundary dissipation $\mathfrak D_B(u_h,v_h;a_h)$ defined in \eqref{eq:boundDB}, 
we observe that, when $\sigma=-1$, the boundary dissipation terms involving $|v_h|^2$ and $|\dot v_h|^2$ vanish. This observation motivates the restriction $\sigma\in(-1,1]$ in the analysis of the nonlinear scheme in Theorem~\ref{thm3}. Indeed, the resulting boundary control of $v_h$ and $\dot v_h$ plays an essential role in establishing the contraction property required for the Banach fixed-point argument.
\end{remark}

\subsection{Error analysis of the linear semi-discrete scheme}\label{sec:linear_error}

To analyze the numerical error of the EDG scheme \eqref{eq:linearized_scheme}, we introduce the errors:
\begin{equation}\label{eq:def_e}
e_u := u - u_h,\qquad e_v := v - v_h
\end{equation}
and compare the EDG solution $(u_h, v_h)$ with suitable projections of the exact solution $(u,v)$. Specifically, let $\tilde u_h$ denote the elementwise $H^1$-projection of
$u$ onto $U_h^{q_u}$, and let $\tilde v_h$ denote the elementwise $L^2$-projection of $v$ onto $V_h^{q_v}$, where the approximation spaces $(U_h^{q_u}, V_h^{q_v})$ are defined in \eqref{eq:approximation_space}. Throughout the error analysis, we assume that
\begin{equation}\label{eq:degree}
q_u-2 \leq q_v \leq q_u.
\end{equation}
We emphasize that the restriction (\ref{eq:degree}) is imposed solely for the purpose of deriving the error estimate. No corresponding constraint on the polynomial degrees $q_u$ and $q_v$ is required for the stability of the scheme. More precisely, on each element $\Omega_j$, the projections
$(\tilde u_h,\tilde v_h)\in U_h^{q_u}\times V_h^{q_v}$ are defined by
\begin{equation}\label{eq:L2_projection}
\int_{\Omega_j} c^2\nabla \phi_u \cdot \nabla \delta_u\, d{\bf x} = 0, \quad \int_{\Omega_j} \phi_v \delta_v\, d{\bf x} = 0,\quad \int_{\Omega_j} \delta_u \,d{\bf x} = 0 \qquad \forall (\phi_u,\phi_v)\in \mathcal P^{q_u}(\Omega_j)\times\mathcal P^{q_v}(\Omega_j),
\end{equation}
where $\delta_u := \tilde{u}_h - u$ and $\delta_v := \tilde{v}_h - v$ denote the corresponding projection errors. We further introduce $\tilde{e}_u := \tilde{u}_h - u_h, \tilde{e}_v := \tilde{v}_h - v_h$, so that 
\begin{equation}\label{eq:decompostion}
    e_u = \tilde e_u - \delta_u,\qquad  e_v = \tilde e_v - \delta_v,
\end{equation} 
and define the numerical error energy as 
\begin{equation}\label{eq:linear_error_energy}
\mathcal{E}_h^{\rm lin}(t) := \frac{1}{2} \int_{\Omega} a_h \big(|\tilde e_v|^2 + |\dot{\tilde e}_v|^2\big) +  c^2 \big(|\nabla_h \tilde e_u|^2 + |\nabla_h \dot{\tilde e}_u|^2\big)\,d{\bf x}.
\end{equation}
Since both the continuous solution $(u, v; a)$ and the EDG solution $(u_h, v_h; a_h)$ satisfy \eqref{eq:linearized_scheme}, subtracting the two formulations and using \eqref{eq:def_e} yields
\begin{subequations}\label{eq:error_aux1}
\begin{align}
\int_{\Omega_j} c^2\nabla\phi_u \cdot \nabla \dot e_u\, d{\bf x} + \mathcal K_j(u,v;\phi_u;a) - \mathcal K_j(u_h,v_h;\phi_u;a_h) &= 0, \label{eq:error_aux1_1}\\
\int_{\Omega_j} {\phi_v} a\dot v - \phi_va_h\dot v_h \,d{\bf x} +\mathcal{H}_j(u,v;\phi_v;a) - \mathcal{H}_j(u_h,v_h;\phi_v;a_h) &= 0.\label{eq:error_aux1_2}
\end{align}
\end{subequations}
Invoking the definitions of $\mathcal K$ and $\mathcal H$ in \eqref{eq:defKH_total}, and the definitions of  $\mathcal K_j$ and $\mathcal H_j$ in \eqref{eq:def_KH}, we have
\begin{equation*}
\begin{aligned}
&\sum_j \mathcal K_j(u,v;\phi_u;a) - \mathcal K_j(u_h,v_h;\phi_u;a_h) \\
= &\  \mathcal K(e_u,e_v;\phi_u;a_h) - \int_{\partial\Omega} c^2\nabla_h \phi_u\cdot{\bf n}\bigg(\frac{(1-\sigma) v}{2} \frac{\sqrt{a_h}-\sqrt{a}}{\sqrt{a_h}}\bigg)  \, dS,
\end{aligned}
\end{equation*}
and
\begin{equation*}
\begin{aligned}
&\sum_j \mathcal H_j(u,v;\phi_v;a) - \mathcal H_j(u_h,v_h;\phi_v;a_h) \\
= &\  \mathcal H(e_u,e_v;\phi_v;a_h) - \int_{\partial\Omega} \alpha\nabla_h\phi_v \cdot{\bf n}\bigg(\frac{(1-\sigma) v}{2} \frac{\sqrt{a_h}-\sqrt{a}}{\sqrt{a_h}}\bigg) + \phi_v \bigg(\frac{(1+\sigma)cv}{2}(\sqrt{a_h} - \sqrt{a})\bigg) \, dS,
\end{aligned}
\end{equation*}
where the additional physical-boundary integrals arise from the dependence of the physical boundary fluxes \eqref{eq:flux_bc_linear} on $a(t,{\bf x})$ and $a_h(t, {\bf x})$. For presentation simplicity, we will denote by
\begin{equation}\label{eq:def_wv}
\omega := v(\sqrt{a_h}-\sqrt{a})/(\sqrt{a_h}),\qquad \nu := v(\sqrt{a_h} - \sqrt{a}),
\end{equation}
for the rest of the paper.

Choosing $(\phi_u,\phi_v)=(\tilde e_u,\tilde e_v)$ in \eqref{eq:error_aux1}, summing over all elements, and combining with the decomposition \eqref{eq:decompostion}, we get
\begin{equation}\label{eq:error_linear_aux2}
\begin{aligned}
&\int_{\Omega} c^2 \nabla_h \tilde e_u \cdot \nabla_h \dot{\tilde e}_u + a_h\tilde e_v\dot{\tilde e}_v \,d{\bf x} + \mathcal{K}(\tilde e_u, \tilde e_v;\tilde e_u;a_h) + \mathcal H(\tilde e_u, \tilde e_v;\tilde e_v;a_h) \\
= & \int_{\Omega} c^2 \nabla_h \tilde e_u \cdot \nabla_h \dot{\delta}_u + a_h\tilde e_v\dot{\delta}_v - \tilde e_v (a-a_h)\dot v\,d{\bf x}+ \mathcal K(\delta_u,\delta_v;\tilde e_u;a_h) + \mathcal H(\delta_u,\delta_v;\tilde e_v;a_h) + \mathcal A(\tilde e_u, \tilde e_v, \omega, \nu),
\end{aligned}
\end{equation}
where $(\omega, \nu)$ are defined in \eqref{eq:def_wv}, and 
\begin{equation*}\label{def:A}
\mathcal A(\tilde e_u, \tilde e_v, \omega, \nu) = \frac{1}{2}\int_{\partial\Omega} (c^2\nabla_h \tilde e_u\cdot{\bf n} +\alpha\nabla_h\tilde e_v\cdot{\bf n})(1-\sigma) \omega + \tilde e_v(1+\sigma)c\nu\, dS.
\end{equation*}
Proceeding as in the derivation of \eqref{eq:KH}, we obtain from \eqref{eq:error_linear_aux2} that
\begin{equation}\label{eq:error_linear_aux7}
\begin{aligned}
&\int_{\Omega} c^2 \nabla_h \tilde e_u \cdot \nabla_h \dot{\tilde e}_u + a_h\tilde e_v\dot{\tilde e}_v + \alpha|\nabla_h \tilde e_v|^2\,d{\bf x} + \mathcal D_I(\tilde e_u, \tilde e_v) + \mathcal D_B(\tilde e_u,\tilde e_v;a_h) \\
= & \int_{\Omega} c^2 \nabla_h \tilde e_u \cdot \nabla_h \dot{\delta}_u + a_h\tilde e_v\dot{\delta}_v - \tilde e_v (a-a_h)\dot v\,d{\bf x} + \mathcal K(\delta_u,\delta_v;\tilde e_u;a_h) + \mathcal H(\delta_u,\delta_v;\tilde e_v;a_h) + \mathcal A(\tilde e_u, \tilde e_v, \omega, \nu).
\end{aligned}
\end{equation}
Similarly, differentiating \eqref{eq:error_aux1} with respect to time, summing the resulting equations over all elements, and using the decomposition \eqref{eq:decompostion} gives
\begin{equation}\label{eq:error_linear_aux3}
\begin{aligned}
&\int_\Omega c^2\nabla_h \phi_u \cdot {\nabla_h\ddot{\tilde e}_u} + \phi_v \dot a \dot v + \phi_v a\ddot v - \phi_v \dot a_h\dot v_h - \phi_v a_h\ddot v_h \,d{\bf x} + \dot{\mathcal K}(\tilde e_u,\tilde e_v;\phi_u;a_h) + \dot{\mathcal H}(\tilde e_u,\tilde e_v;\phi_v;a_h)  \\
= & \int_\Omega c^2 \nabla_h\phi_u \cdot {\nabla_h \ddot\delta_u} \,d{\bf x} + \dot{\mathcal K}(\delta_u,\delta_v;\phi_u;a_h) + \dot{\mathcal H}(\delta_u,\delta_v;\phi_v;a_h) + \mathcal A(\phi_u,\phi_v,\dot \omega,\dot \nu).
\end{aligned}
\end{equation}
To proceed, we first compute that
\begin{equation}\label{eq:error_linear_aux4}
\begin{aligned}
&\int_\Omega \phi_v \dot a \dot v + \phi_v a\ddot v - \phi_v \dot a_h\dot v_h - \phi_v a_h\ddot v_h \,d{\bf x} \\
= & \int_\Omega \phi_v\dot a_h(\dot{\tilde e}_v - \dot\delta_v) + \phi_v(\dot a - \dot a_h)\dot v + \phi_v a_h (\ddot{\tilde e}_v - \ddot\delta_v) + \phi_v(a-a_h)\ddot v\,d{\bf x}.
\end{aligned}
\end{equation}
Substituting \eqref{eq:error_linear_aux4} into \eqref{eq:error_linear_aux3}, choosing $(\phi_u,\phi_v) = (\dot{\tilde e}_u, \dot{\tilde e}_v)$, and following the same derivation as in \eqref{eq:dot_KH}, we have that
\begin{equation}\label{eq:error_linear_aux6}
\begin{aligned}
&\int_\Omega c^2 \nabla_h \dot{\tilde e}_u \cdot \nabla_h {\ddot{\tilde e}_u}+a_h\dot{\tilde e}_v \ddot{\tilde e}_v + \dot a_h|\dot{\tilde e}_v|^2 + \alpha |\nabla_h \dot{\tilde e}_v|^2 \,d{\bf x} + \mathcal D_I(\dot{\tilde e}_u, \dot{\tilde e}_v) + \mathcal D_B(\dot{\tilde e}_u,\dot{\tilde e}_v;a_h) + \mathfrak E_B(\tilde e_u,\tilde e_v;a_h) \\
= & \int_\Omega c^2\nabla_h \dot{\tilde e}_u \cdot \nabla_h {\ddot\delta_u} + \dot a_h\dot{\tilde e}_v\dot\delta_v +a_h\dot{\tilde e}_v\ddot\delta_v- \dot{\tilde e}_v(\dot a - \dot a_h)\dot v - \dot{\tilde e}_v(a-a_h)\ddot v\, d{\bf x} \\
& + \dot{\mathcal K}(\delta_u,\delta_v;\phi_u;a_h)\big|_{\phi_u = \dot{\tilde e}_u} + \dot{\mathcal H}(\delta_u,\delta_v;\phi_v;a_h)\big|_{\phi_v = \dot{\tilde e}_v} + \mathcal A(\dot{\tilde e}_u,\dot{\tilde e}_v,\dot \omega,\dot \nu).
\end{aligned}
\end{equation}
Now, summing \eqref{eq:error_linear_aux7} and \eqref{eq:error_linear_aux6} together, and invoking the definition of error energy in \eqref{eq:linear_error_energy} and the projection defined in \eqref{eq:L2_projection} yields
\begin{equation}\label{eq:time_derivative_error}
\begin{aligned}
&\dot{\mathcal E}_h^{\rm lin}(t) + \int_\Omega \alpha \big(|\nabla_h \tilde e_v|^2 + |\nabla_h \dot{\tilde e}_v|^2\big)\,d{\bf x} + \mathfrak D_I(\tilde e_u, \tilde e_v) + \mathfrak D_B(\tilde e_u,\tilde e_v;a_h) \\
= & \underbrace{\int_{\Omega} \frac{1}{2}\dot a_h \big(|\tilde e_v|^2 -  |\dot{\tilde e}_v|^2\big)+a_h\tilde e_v\dot{\delta}_v  + \dot a_h\dot{\tilde e}_v\dot\delta_v +a_h\dot{\tilde e}_v\ddot\delta_v - (a-a_h)( \tilde e_v \dot v + \dot{\tilde e}_v \ddot v) - (\dot a - \dot a_h)\dot{\tilde e}_v\dot v \,d{\bf x}}_{\mathcal V_0} \\
& + \underbrace{\mathcal K(\delta_u,\delta_v;\tilde e_u;a_h) + \mathcal H(\delta_u,\delta_v;\tilde e_v;a_h) +\dot{\mathcal K}(\delta_u,\delta_v;\phi_u;a_h)\big|_{\phi_u = \dot{\tilde e}_u} + \dot{\mathcal H}(\delta_u,\delta_v;\phi_v;a_h)\big|_{\phi_v = \dot{\tilde e}_v}}_{\mathcal V_1}  \\
&+ \mathcal A(\tilde e_u, \tilde e_v, \omega, \nu)+ \mathcal A(\dot{\tilde e}_u,\dot{\tilde e}_v,\dot \omega,\dot \nu),
\end{aligned}
\end{equation}
where $\mathfrak D_I \geq 0,\mathfrak D_B\geq 0$ are defined in \eqref{eq:boundDI} and \eqref{eq:boundDB}, respectively.

We are now in a position to establish the error estimate for the linearized scheme, which is stated in the following theorem.

\begin{truth}\label{thm2}
Let the assumptions of Theorem~\ref{thm1} hold. In particular, $a_h\in \mathcal A_{\rm adm}$, where $\mathcal A_{\rm adm}$ is defined in \eqref{eq:admissible_set}. Let $\bar q = \min\{q_u, q_v\} > \frac{d}{2}-1$, where the polynomial degrees $q_u,q_v$ satisfy \eqref{eq:degree}. Assume that the exact solution $(u,v)$ satisfies
\[\begin{aligned}
    u\in L^\infty(0,T;H^{q_u+1}(\Omega)) \cap H^1(0,T;H^{q_u+1}(\Omega)),\quad
    v\in W^{1,\infty}(0,T;H^{q_v+1}(\Omega))\cap H^2(0,T;H^{q_v+1}(\Omega)).
    \end{aligned}
\]
Then there exists a positive constant $C_{\rm lin} > 0$, independent of mesh size $h$, depending only on the polynomial degrees $q_u, q_v$; the flux parameters $\theta,\lambda_u,\lambda_v,\tau,\sigma$; the coefficients $c,\alpha$ and the constant $a_0$ in \eqref{eq:a_condition}; the bounds associated with the coefficient $a_h$, namely, $a_\ast,a^\ast,M_a$ and $\mu$, and the shape regularity of the mesh, and depending at most linearly on the final time $T$, such that, for every $t\in[0,T]$,
\begin{equation}\label{eq:thm2}
\begin{aligned}
 \max_{s\in[0,t]}\Big(&\|\nabla_h e_u\|_{L^2(\Omega)}^2  + \|e_v\|_{L^2(\Omega)}^2 + \|\nabla_h \dot e_u\|_{L^2(\Omega)}^2 + \|\dot e_v\|_{L^2(\Omega)}^2 \Big) \\
+ &\  \chi_\sigma \big(\|e_v\|_{L^2(0,t;L^2(\partial\Omega))}^2 +  \|\dot e_v\|_{L^2(0,t;L^2(\partial\Omega))}^2\big)
  \leq   C_{\rm lin}\big( \mathcal E_h^{\rm lin}(0) + h^{2\zeta} \mathcal{R}(t) + \mathcal Q(t)\big),
  \end{aligned}
\end{equation}
where $e_u$ and $e_v$ are defined in \eqref{eq:def_e}, and 
\[\chi_\sigma = \left\{\begin{aligned}
&1, \quad \sigma \in (-1,1] \\
&0, \quad \sigma= -1
\end{aligned}\right. ; \qquad\qquad \zeta = \left\{\begin{aligned}
\bar q-1,\quad &\lambda_u+\lambda_v= 0 \textup{ or } \tau = 0 \textup{ or } \sigma = {-1} \\
\bar q-\frac{1}{2},\quad &\lambda_u+\lambda_v> 0, \tau > 0, \sigma \in (-1,1{]}
\end{aligned}\right..\]
Moreover,
\begin{equation}\label{eq:Q_def}
\begin{aligned}
&{\mathcal{R}(t) :=  \max_{s\in[0,t]}\Big( |u|_{H^{q_u+1}(\Omega)}^2 + |v|_{H^{q_v+1}(\Omega)}^2 + |\dot v|_{H^{q_v+1}(\Omega)}^2\Big)}\\
& \hspace{1.2cm} + \int_0^t |u|_{H^{q_u+1}(\Omega)}^2 + | v|_{H^{q_v+1}(\Omega)}^2 + |\dot v|_{H^{q_v+1}(\Omega)}^2 + |\ddot v|_{H^{q_v+1}(\Omega)}^2\,ds,\\
&\mathcal Q(t) := {\big(\|a - a_h\|_{L^\infty(0,t;L^2(\Omega))}^2 +\|\dot a - \dot a_h\|_{L^\infty(0,t;L^2(\Omega))}^2\big)}  \int_0^t \|\dot v\|_{L^\infty(\Omega)}^2 +  \|\ddot v\|_{L^\infty(\Omega)}^2 \,ds\\
  & \hspace{1.2cm} +  \int_0^t\|a - a_h\|_{L^2(\partial\Omega)}^2 + \|\dot a - \dot a_h\|_{L^2(\partial\Omega)}^2\, ds \max_{s\in[0,t]}\big(\|v\|_{L^\infty(\Omega)}^2 + \|\dot v\|_{L^\infty(\Omega)}^2\big). 
        \end{aligned}
\end{equation}

\end{truth}

\begin{proof}
Invoking the projections defined in \eqref{eq:L2_projection}, by the Bramble--Hilbert lemma (see, e.g., \cite{ciarlet2002finite}), we have, for each element $\Omega_j$,
\begin{equation}\label{lemma_L2}
\begin{aligned}
&\|\delta_u\|_{L^2(\Omega_j)}^2 \leq Ch^{2q_u+2}|u(t,\cdot)|_{H^{q_u+1}(\Omega_j)}^2,\quad \Vert \dot\delta_u\Vert_{L^2(\Omega_j)}^2 \leq Ch^{2\bar q+2}| \dot u(t,\cdot)|_{H^{\bar q+1}(\Omega_j)}^2,\\
&  \|\delta_v\|_{L^2(\Omega_j)}^2 \leq Ch^{2q_v+2} |v(t,\cdot)|_{H^{q_v+1}(\Omega_j)}^2, \quad\  \Vert \dot\delta_v\Vert_{L^2(\Omega_j)}^2 \leq Ch^{2q_v+2}| \dot v(t,\cdot)|_{H^{q_v+1}(\Omega_j)}^2.
\end{aligned}
\end{equation}
To establish \eqref{eq:thm2}, we begin by estimating $\mathcal V_0$ on the right-hand side of \eqref{eq:time_derivative_error}. Using the Cauchy--Schwarz inequality and Young's inequality yields
\[  \begin{aligned}
    \mathcal{V}_0 
     \leq  & \  \int_{\Omega}\frac{1}{2} \Big|\frac{\dot a_h}{a_h}\Big|(a_h|\tilde e_v|^2  + a_h|\dot{\tilde e}_v|^2)+ \frac{{a^\ast}}{\varepsilon_1}\Big|\frac{\dot a_h}{a_h}\Big||\dot \delta_v|^2 +  \frac{\varepsilon_1}{4}\Big|\frac{\dot a_h}{a_h}\Big|a_h|\dot{\tilde e}_v|^2 \, d{\bf x} \\
     & + \Big(\int_\Omega a_h |\tilde e_v|^2 + a_h|\dot{\tilde e}_v|^2 \,d{\bf x}\Big)^{1/2}\Big(\int_\Omega a_h|\dot\delta_v|^2 + a_h|\ddot\delta_v|^2\, d{\bf x}\Big)^{1/2}  \\
    & + \|(a-a_h)\dot v\|_{L^2(\Omega)}\|\tilde e_v\|_{L^2(\Omega)} + \|(\dot a - \dot a_h)\dot v\|_{L^2(\Omega)}\|\dot{\tilde e}_v\|_{L^2(\Omega)} +  \|(a - a_h)\ddot v\|_{L^2(\Omega)}\|\dot{\tilde e}_v\|_{L^2(\Omega)},
    \end{aligned}
\]
where $\varepsilon_1 > 0$ is the parameter arising from Young's inequality. Invoking the definition of the error energy \eqref{eq:linear_error_energy}, the projection errors given in \eqref{lemma_L2}, and the fact that $a_h\in \mathcal A_{\rm adm}$, where $\mathcal A_{\rm adm}$ is defined in \eqref{eq:admissible_set}, we further obtain
\begin{equation}\label{eq:liner_error_aux2}
\begin{aligned}
\mathcal{V}_0  \leq  & \ \Big(1 + \frac{\varepsilon_1}{2}\Big)\Big\|\frac{\dot a_h}{a_h}\Big\|_{L^\infty(\Omega)}\mathcal{E}_h^{\rm lin}+ C\Big\|\frac{\dot a_h}{a_h}\Big\|_{L^\infty(\Omega)}h^{2\bar{q}+2}|\dot v|_{H^{q_v+1}(\Omega)}^2 +  C\sqrt{\mathcal E_h^{\rm lin}} h^{\bar q}\Big( |\dot v|^2_{H^{q_v+1}(\Omega)}  + |\ddot v|^2_{H^{q_v+1}(\Omega)}\Big)^{1/2} \\
& + C\sqrt{\mathcal E_h^{\rm lin}}\Big(\|(a - a_h)\dot v\|_{L^2(\Omega)} + \|(\dot a - \dot a_h)\dot v\|_{L^2(\Omega)} + \|(a - a_h)\ddot v\|_{L^2(\Omega)}\Big).
\end{aligned}
\end{equation}

To estimate $\mathcal A(\tilde e_u, \tilde e_v, \omega, \nu)+ \mathcal A(\dot{\tilde e}_u,\dot{\tilde e}_v,\dot \omega,\dot \nu)$ in \eqref{eq:time_derivative_error}, by the nondegeneracy assumption on $a$ in \eqref{eq:a_condition} and the fact that $a_h \in \mathcal A_{\rm adm}$, together with the definitions of $\omega, \nu$ in \eqref{eq:def_wv}, we have \[ |\omega| + |\nu| \lesssim |a_h - a||v|, \qquad |\dot \omega| + |\dot\nu|  \lesssim |\dot a_h -\dot a||v|+ |a_h - a|(|v| + |\dot v|).\]
Therefore, applying Young's inequality with a parameter $\varepsilon_2>0$ gives
\begin{equation}\label{eq:bounds_physic}
\begin{aligned}
&\mathcal A(\tilde e_u, \tilde e_v, \omega, \nu)+ \mathcal A(\dot{\tilde e}_u,\dot{\tilde e}_v,\dot \omega,\dot \nu)\\
= &\int_{\partial\Omega}\frac{1-\sigma}{2}(\omega w_h + \dot \omega \dot w_h) + \frac{(1+\sigma)c\sqrt{a_h}}{2}(\omega \tilde e_v+\dot\omega \dot{\tilde e}_v) + \frac{(1+\sigma)c\sqrt{a_h}}{4}\frac{\dot a_h}{a_h}\,\omega\,\dot{\tilde e}_v\,dS \\
\leq &\  \varepsilon_2 \int_{\partial\Omega} \frac{(1+\sigma)c\sqrt{a_h}}{2}\big(|\tilde e_v|^2 + |\dot{\tilde e}_v|^2\big) + \frac{1-\sigma}{2c\sqrt{a_h}}\big(|w_h|^2 + |\dot w_h|^2\big)\,dS\\
&+C\int_{\partial\Omega}
|a_h-a|^2\big(|v|^2+|\dot{v}|^2\big)
+|\dot{a}_h-\dot{a}|^2|v|^2\,dS,\qquad \text{where } w_h \!:= \!{(c^2\nabla_h \tilde e_u + \alpha \nabla_h \tilde e_v)\cdot{\bf n}}.
\end{aligned}
\end{equation}

Finally, it remains to estimate $\mathcal V_1$ in \eqref{eq:time_derivative_error}, we divide our arguments into two cases.

\textbf{Case I:} $\lambda_u+\lambda_v = 0$ or $\tau = 0$ or $\sigma = {-1}$. In this case, applying the projection definition \eqref{eq:L2_projection}, the Cauchy--Schwarz inequality and Young's inequality to the volume integrals, and applying the trace inequality together with the Cauchy--Schwarz inequality and Young's inequality to the boundary integrals yields
\begin{equation}\label{eq:linear_boundary}
\begin{aligned}
\mathcal V_1
\leq &\ \frac{1}{2}\int_{\Omega}\alpha\big(|\nabla_h\tilde e_v|^2 +  |\nabla_h\dot{\tilde e}_v|^2\big)\, d\mathbf{x}  + Ch^{2\bar q-2} \Big(|u|_{H^{q_u+1}(\Omega)}^2
        + |v|_{H^{q_v+1}(\Omega)}^2 
+|\dot v|_{H^{q_v+1}(\Omega)}^2 \Big)\\
&+  C \sqrt{\mathcal E_h^{\rm lin}}h^{{\bar q} -1} \Big(|u|_{H^{q_u+1}(\Omega)}^2+|v|_{H^{q_v+1}(\Omega)}^2 +
|\dot v|_{H^{q_v+1}(\Omega)}^2 \Big)^{1/2}.
\end{aligned}
\end{equation}

Substituting \eqref{eq:liner_error_aux2}--\eqref{eq:linear_boundary} into \eqref{eq:time_derivative_error}, choosing $\varepsilon_1$ in \eqref{eq:liner_error_aux2} and $\varepsilon_2$ in \eqref{eq:bounds_physic} such that
\[\Big(1+\frac{\varepsilon_1}{2}\Big)M_a<1, \qquad \varepsilon_2\le\frac{1}{2}\Big(1-\frac{\mu}{4}\Big),\]
and solving the resulting differential inequality, together with the definition of the error energy \eqref{eq:linear_error_energy} and the fact that $a_h\in \mathcal A_{\rm adm}$, 
we obtain
\begin{equation}\label{eq:error_linear1}
{
\begin{aligned}
&\max_{s\in[0,t]}\mathcal{E}_h^{\rm lin}(s) + \frac{1}{2}\int_0^t \Big(\int_\Omega \alpha\big(|\nabla_h\tilde e_v|^2 + |\nabla_h\dot{\tilde e}_v|^2\big)\,d{\bf x} + \mathfrak D_I(\tilde e_u,\tilde e_v)  + \mathfrak D_B(\tilde e_u, \tilde e_v;a_h)\Big)\, ds \\
\leq &\ C\Big(\mathcal E_h^{\rm lin}(0) + (1+t)\big(h^{2\bar q-2} \mathcal{R}(t) + \mathcal Q(t)\big)\Big),
\end{aligned}}
\end{equation}
where $\mathcal{R}(t), \mathcal Q(t)$ are defined in \eqref{eq:Q_def}.
 
\textbf{Case II:} $\lambda_u + \lambda_v > 0$, $\tau > 0$, and $\sigma \in (-1,1{]}$. In this case, both the interelement and physical-boundary dissipation terms are strictly positive. This additional dissipation allows the boundary integrals in $\mathcal V_1$ to be estimated more sharply than in \textbf{Case I}. In particular, we have
\begin{equation}\label{eq:BC_shaper}
\begin{aligned}
\mathcal V_1 \leq &\ {\frac{1}{2}}\int_{\Omega} \alpha\big(|\nabla_h\tilde e_v|^2 +  |\nabla_h\dot{\tilde e}_v|^2\big)\, d\mathbf{x} +Ch^{2\bar q}\Big(|v|_{H^{q_v+1}(\Omega)}^2 + |\dot v|_{H^{q_v+1}(\Omega)}^2\Big) \\
&+ C\sqrt{\mathcal E_{h}^{\rm lin}}h^{\bar q}\Big(|v|_{H^{q_v+1}(\Omega)}^2 + |\dot v|_{H^{q_v+1}(\Omega)}^2\Big)^{1/2} \\
&+ Ch^{2\bar q-1}\Big(|u|_{H^{q_u+1}(\Omega)}^2
        + |v|_{H^{q_v+1}(\Omega)}^2  
+|\dot v|_{H^{q_v+1}(\Omega)}^2 \Big) \\
& + \frac{1}{2}\mathfrak D_I(\tilde e_u,\tilde e_v)+ \varepsilon_3\Big(1-\frac{\mu}{4}\Big)\int_{\partial\Omega} \frac{(1+\sigma)c\sqrt{a_h}}{2}\big(|\tilde e_v|^2+ |\dot{\tilde e}_v|^2\big)\,dS
\\
& + \varepsilon_3\Big(1-\frac{\mu}{4}\Big)\int_{\partial\Omega}\frac{1-\sigma}{2c\sqrt{a_h}}\Big(\big|\big(c^2\nabla_h\tilde e_u+\alpha\nabla_h\tilde e_v\big)\cdot{\bf n}\big|^2 + \big|\big(c^2\nabla_h\dot{\tilde e}_u+\alpha\nabla_h\dot{\tilde e}_v\big)\cdot{\bf n}\big|^2\Big)\,dS,
\end{aligned}
\end{equation}
where the first two lines follow from Young's inequality and the Cauchy--Schwarz inequality applied to the volume integrals in $\mathcal V_1$; while the last three lines are derived by applying Young's inequality with a parameter $\varepsilon_3 > 0$ to the boundary integrals in $\mathcal V_1$.


Choosing $\varepsilon_1$ in \eqref{eq:liner_error_aux2}, $\varepsilon_2$ in \eqref{eq:bounds_physic} and $\varepsilon_3$ in \eqref{eq:BC_shaper} such that
\[\Big(1+\frac{\varepsilon_1}{2}\Big)M_a<1,\qquad \frac{\varepsilon_2}{1-\mu/4} + \varepsilon_3 \le \frac{1}{2},\]
substituting \eqref{eq:BC_shaper}, together with \eqref{eq:liner_error_aux2}--\eqref{eq:bounds_physic}, into \eqref{eq:time_derivative_error}, invoking \eqref{eq:boundDB} and solving the resulting differential inequality, we obtain 
\begin{equation}\label{eq:error_linear2}
{
\begin{aligned}
&\max_{s\in[0,t]}\mathcal{E}_h^{\rm lin}(s) + \frac{1}{2}\int_0^t \Big(\int_\Omega \alpha\big(|\nabla_h\tilde e_v|^2 + |\nabla_h\dot{\tilde e}_v|^2\big)\,d{\bf x} + \mathfrak D_I({\tilde e}_u,{\tilde e}_v) + \mathfrak D_B(\tilde e_u,\tilde e_v;a_h)\Big)\, ds \\
\leq  & \ C\Big(\mathcal E_h^{\rm lin}(0)+ (1+t)\big(h^{2\bar{q}-1}\mathcal{R}(t) + \mathcal Q(t)\big)\Big).
\end{aligned}}
\end{equation}
Finally, since $e_u = \tilde e_u - \delta_u$, $e_v = \tilde e_v - \delta_v$ and $\mathfrak D_I \geq 0,\mathfrak D_B\geq 0$, the triangle inequality, together with projection errors \eqref{lemma_L2} and \eqref{eq:error_linear1}--\eqref{eq:error_linear2}, yields \eqref{eq:thm2}.

\end{proof}

We next turn to the nonlinear semi-discrete scheme
\eqref{eq:EDG}--\eqref{eq:mean_value}. The analysis will build on the error estimate for a class of linearized problems established in Theorem~\ref{thm2}, together with a fixed-point argument.

\section{Error analysis of the nonlinear semi-discrete scheme}\label{sec:nonlinear}

Before stating our main convergence result, we first establish a lemma concerning the initialization of the EDG scheme \eqref{eq:EDG}--\eqref{eq:mean_value}.

\begin{lemma}\label{lemma:nonlinear_IC}
Let $(u,v)$ be a sufficiently regular solution of
the Westervelt equation \eqref{eq:westervelt}; see Theorem~\ref{thm3}. Set $u^0:=u(0,\cdot), v^0:=v(0,\cdot), \dot v^0 := \dot v(0,\cdot), a^0:=1-2\beta v^0$, and assume that
\[(u^0,v^0,\dot v^0) \in H^{q_u+1}(\Omega) \times H^{q_v+1}(\Omega) \times H^{q_v+1}(\Omega);\qquad 0<a_0\le a^0({\bf x})\le a_1
\qquad\textup{in }\Omega.\]
Let $\kappa_u, \kappa_v > 0$ be fixed constants independent of $h$. Suppose that the polynomial degrees $q_u,q_v$ satisfy \eqref{eq:degree}{, and that the dissipation parameters in the numerical fluxes satisfy $\lambda_u,\lambda_v,\tau\ge 0$ and $\sigma\in(-1,1]$}, and let $(u_{h}^0,v_{h}^0) \in U_h^{q_u}\times V_h^{q_v}$, where the approximation spaces $U_h^{q_u}$ and $V_h^{q_v}$ are defined in \eqref{eq:approximation_space}, be defined by requiring that, for each element $\Omega_j$ and all $(\phi_u,\phi_v)\in U_h^{q_u}\times V_h^{q_v}$,
\begin{subequations}\label{eq:nonlinear_Ritz_IC}
\begin{align}
\kappa_u \int_{\Omega_j}
c^2\nabla\phi_u\cdot\nabla u_{h}^0\,d{\bf x} + \mathcal K_j (u_{h}^0,v_{h}^0;\phi_u;a_{h}^0) &= \kappa_u \int_{\Omega_j} c^2\nabla\phi_u\cdot\nabla u^0\,d{\bf x} - \int_{\Omega_j} c^2\nabla\phi_u\cdot\nabla v^0\,d{\bf x},\label{eq:nonlinear_Ritz_IC_u}\\
\kappa_v \int_{\Omega_j} \phi_v v_{h}^0\,d{\bf x} + \mathcal H_j(u_{h}^0,v_{h}^0;\phi_v;a_{h}^0)&=\kappa_v
\int_{\Omega_j}\phi_v v^0\,d{\bf x}-\int_{\Omega_j}\phi_v a^0\dot v^0\,d{\bf x},\label{eq:nonlinear_Ritz_IC_v}\\
\int_{\Omega_j}(u_{h}^0-u^0)\,d{\bf x}&=0,
\label{eq:nonlinear_Ritz_IC_mean}
\end{align}
\end{subequations}
where $a_{h}^0:=1-2\beta v_{h}^0$, and $\mathcal K_j$ and $\mathcal H_j$ are defined in \eqref{eq:def_KH}.
 
Let $\bar q:=\min\{q_u,q_v\}$, assume further that
\begin{equation}\label{eq:barq}
\bar q >
\begin{cases}
\dfrac{d+2}{2},
&
\lambda_u+\lambda_v=0
\ \textup{or}\ %
\tau=0,
\\[2mm]
\dfrac{d+1}{2},
&
\lambda_u+\lambda_v>0,\quad
\tau>0,
\end{cases}
\end{equation}
where $\lambda_u,\lambda_v$ and $\tau$ are the parameters appearing in the numerical fluxes. Define
\begin{equation*}\label{eq:def_R0}
\mathcal R_0 := |u^0|_{H^{q_u+1}(\Omega)}^2+|v^0|_{H^{q_v+1}(\Omega)}^2 +|\dot v^0|_{H^{q_v+1}(\Omega)}^2,
\end{equation*}
and $\mathcal M_0 :=\mathcal R_0 + \|\dot v^0\|_{L^\infty(\Omega)}^2+\|v^0\|_{L^\infty(\Omega)}^2+ \|\nabla u^0\|_{L^\infty(\Omega)}^2 + \|\nabla v^0\|_{L^\infty(\Omega)}^2$. Then, provided that $h$ and $\mathcal M_0$ are sufficiently small, the nonlinear system \eqref{eq:nonlinear_Ritz_IC} admits a unique solution $(u_{h}^0,v_{h}^0) \in U_h^{q_u}\times V_h^{q_v}$. Moreover, defining $e_u^0 = u^0 - u_h^0, e_v^0 = v^0 - v_h^0$, and denoting the corresponding errors in the time derivatives by $\dot e_u^0:=\dot u^0-\dot u_h^0, \dot e_v^0:=\dot v^0-\dot v_h^0$, the following initial error estimate holds:
\begin{equation}\label{eq:convergence_initial}
\|\nabla_h e_u^0\|_{L^2(\Omega)}^2
+\|e_v^0\|_{L^2(\Omega)}^2 + \|\nabla_h \dot e_u^0\|_{L^2(\Omega)}^2 + \|\dot e_v^0\|_{L^2(\Omega)}^2 \leq C_{\rm init} h^{2\zeta}\mathcal R_0,
\end{equation}
where $C_{\rm init}$ is independent of $h$, and
\begin{equation}\label{eq:zeta}
\zeta = \left\{\begin{aligned}
\bar q-1,\quad &\lambda_u+\lambda_v= 0 \textup{ or } \tau = 0, \\
\bar q-\frac{1}{2},\quad &\lambda_u+\lambda_v> 0, \tau > 0.
\end{aligned}\right.
\end{equation}

\end{lemma}

\begin{proof}
The proof is deferred to the Appendix.
\end{proof}

We are now in a position to state the main convergence theorem.

\begin{truth}\label{thm3}
Let $c({\bf x}) > 0$ and $\alpha({\bf x}) \geq 0$, and suppose that the dissipation parameters in the numerical fluxes \eqref{eq:generalfl} and \eqref{eq:flux_bc} satisfy $\lambda_u,\lambda_v,\tau \geq 0$ and $\sigma\in(-1,1]$. Assume that the polynomial degrees $q_u,q_v$ satisfy \eqref{eq:degree} and that $\bar q := \min\{q_u, q_v\}$ satisfies \eqref{eq:barq}. Let $(u,v)$, with $v=\dot u$, be the exact solution of \eqref{eq:westervelt}, satisfying
\begin{equation*}
\begin{aligned}
    u\in L^\infty(0,T;H^{q_u+1}(\Omega)) \cap H^1(0,T;H^{q_u+1}(\Omega)),\quad
    v\in W^{1,\infty}(0,T;H^{q_v+1}(\Omega))\cap H^2(0,T;H^{q_v+1}(\Omega)).
    \end{aligned}
\end{equation*}
Fix admissibility constants $0<a_\ast < 1< a^\ast <\infty$, $M_a \in (0,1)$ and $\mu \in (0,4)$ as in Theorems~\ref{thm1}--\ref{thm2}. Let $a:=1-2\beta v$, and assume that the exact coefficient lies strictly inside these admissibility bounds, namely, there exist constants $a_0,a_1, M_a^{\rm ex}$ and $\mu^{\rm ex}$ such that
\begin{equation*}\label{eq:exact_a_bounds_revised}
    0<a_\ast < a_0\le a(t,\mathbf x)\le a_1 < a^\ast < \infty \qquad \forall (t,{\bf x}) \in [0,T]\times \bar\Omega,
\end{equation*}
and
\begin{equation*}\label{eq:exact_a_time_bounds_revised}
    \|\dot a/a\|_{L^1(0,T;L^\infty(\Omega))}
    \le M_a^{\rm ex}<M_a<1,
    \quad
   \|\dot a/a\|_{L^\infty(0,T;L^\infty(\partial\Omega))}
    \le \mu^{\rm ex} < \mu < 4.
\end{equation*}
Let the approximation spaces $(U_h^{q_u}, V_h^{q_v})$ be defined in \eqref{eq:approximation_space}, and initialize the EDG scheme \eqref{eq:EDG}--\eqref{eq:mean_value} by $\big(u_h(0,{\bf x}), v_h(0,{\bf x})\big) =  \big(u_h^0, v_h^0\big)$, where $\big(u_h^0, v_h^0\big)$ is the discrete initial datum constructed in Lemma~\ref{lemma:nonlinear_IC}. Define
\[\begin{aligned}
\mathcal M(T) := &\ \mathcal{R}(T) +  \|v\|_{L^\infty(0,T;L^\infty(\Omega))}^2 + \|\dot v\|_{L^\infty(0,T;L^\infty(\Omega))}^2  \\
 + &\ \|\nabla u\|_{L^\infty(0,T;L^\infty(\Omega))}^2 + \|\nabla v\|_{L^\infty(0,T;L^\infty(\Omega))}^2 
         +  \|\dot v\|_{L^2(0,T;L^\infty(\Omega))}^2 + \|\ddot v\|_{L^2(0,T;L^\infty(\Omega))}^2,
\end{aligned}\]
where $\mathcal R(T)$ is defined in \eqref{eq:Q_def}. Furthermore, {let $\mathcal{S} :=  W^{1,\infty}(0,T;U_h^{q_u})\times W^{1,\infty}(0,T;V_h^{q_v})$ and} define
\[\mathcal X_h := \left\{(w,z) \in \mathcal{S}: \int_{\Omega_j} \dot w - z\,d{\bf x} = 0 \quad \forall \Omega_j\in \mathcal T_h {\textup{ and a.e. } t\in(0,T)}\right\}.\]
{For any pair $(w,z) \in [W^{1,\infty}(0,T;H^1(\mathcal T_h))]^2$, define the energy seminorm}
\[
\|(w,z)\|_{E,h}^2
:=
{\esssup_{t\in(0,T)}}
\Big(
\|\nabla_h w\|_{L^2(\Omega)}^2
+\|z\|_{L^2(\Omega)}^2
+\|\nabla_h \dot w\|_{L^2(\Omega)}^2
+\|\dot z\|_{L^2(\Omega)}^2
\Big)+\int_0^T\int_{\partial \Omega} |z|^2 + |\dot z|^2\, dS\,dt.
\]
Then, provided that $h$ and $\mathcal M(T)$ are sufficiently small, the EDG scheme \eqref{eq:EDG}--\eqref{eq:mean_value} admits a unique solution $(u_h, v_h) \in \mathcal X_h$. Moreover, there exists a constant $C_{\rm non} > 0$, independent of $h$, such that
\[\|(u-u_h, v-v_h)\|_{E,h}^2 \le C_{\rm non} h^{2\zeta} \mathcal{R}(T),
\]
where $\zeta$ is given in \eqref{eq:zeta}.
\end{truth}

\begin{proof}
We prove the result by means of the Banach fixed-point theorem. Define the set
\begin{equation}\label{eq:def_Ball}
\mathfrak B_h:=\left\{\begin{aligned}
(\mathfrak u,\mathfrak v) \in  \mathcal X_h: 
(\mathfrak u(0,{\bf x}),\mathfrak v(0, {\bf x}))=(u_h^0,v_h^0);& \\
\|(u-\mathfrak u,v-\mathfrak v)\|_{E,h}^2
\le C_{\rm non}h^{2\zeta}\mathcal{R}(T)&
\end{aligned}\right\}.
\end{equation}
Here, $(u,v = \dot u)$ denotes the exact solution of the Westervelt equation \eqref{eq:westervelt}, and the discrete initial data $(u_h^0,v_h^0)$ is obtained by Lemma~\ref{lemma:nonlinear_IC}. The constant $C_{\rm non}>0$ will be specified below. We first observe that $\mathfrak{B}_h$ is nonempty. Indeed, consider the solution of the linearized problem \eqref{eq:linearized_scheme} with the exact coefficient $a=1-2\beta v$ and with the initial data $(u_h^0,v_h^0)$. Since the coefficient perturbation vanishes in this case, Theorem~\ref{thm2} and Lemma~\ref{lemma:nonlinear_IC} yield
\[\|(u-\mathfrak u,v-\mathfrak v)\|_{E,h}^2 \leq C_{\rm lin}(1+{C_{\rm init}})h^{2\zeta}\mathcal R(T),\]
{where we have used $\mathcal R_0 \leq \mathcal R(T)$.} Hence, $\mathfrak B_h\neq\emptyset$ provided that $C_{\rm non} \geq C_{\rm lin}(1+{C_{\rm init}})$.

We next verify that $a_{\mathfrak v}:=1-2\beta \mathfrak v \in \mathcal A_{\rm adm}$, where $\mathcal A_{\rm adm}$ is defined in \eqref{eq:admissible_set}, for every $(\mathfrak u, \mathfrak v) \in \mathfrak B_h$. To this end, for fixed $t\in[0,T]$, we can pick an element $\Omega_k \in \mathcal{T}_h$, such that $\max_{{\bf x} \in \bar{\Omega}}|\mathfrak v(t, {\bf x})|=\max_{{\bf x} \in \Omega_k}|\mathfrak v(t, {\bf x})|$. Introducing the projection $\tilde v_h$ defined by \eqref{eq:L2_projection} and applying a local inverse inequality, we obtain
\[\begin{aligned}
\max_{{\bf x} \in \bar{\Omega}}|\mathfrak v(t, {\bf x})|^2 & =\max_{{\bf x} \in \Omega_k}|\mathfrak v(t, {\bf x})|^2 \\
& \leq {3\Big(}\|\tilde v_h - v\|_{L^{\infty}(\Omega_k)}^2 + \|\mathfrak v - \tilde v_h\|_{L^\infty(\Omega_k)}^2 + \|v\|_{L^\infty(\Omega_k)}^2{\Big)}\\
& \leq C\Big(\|\tilde v_h - v\|_{L^\infty(\Omega_k)}^2 + h_k^{-d} \|\mathfrak v - \tilde v_h\|_{L^2(\Omega_k)}^2\Big) + {3}\|v\|_{L^\infty(\Omega_k)}^2.
\end{aligned}\]
Using the approximation properties of the projection, the fact that $(\mathfrak u,\mathfrak v) \in \mathfrak{B}_h$, the triangle inequality, and the quasi-uniformity of the mesh, we further obtain
\[\begin{aligned}
\max_{{\bf x} \in {\bar \Omega}}\left|\mathfrak v(t, {\bf x})\right|^2 & \leq Ch_k^{-d} h^{2\zeta} C_{\rm non} \mathcal{R}(T)+Ch_{k}^{2q_v+2-d}|v|_{H^{q_v+1}(\Omega_k)}^2+{3}\|v\|_{L^\infty(\Omega_k)}^2 \\
& \leq  CC_{\rm non} h^{2\zeta-d}\mathcal{R}(T)+Ch^{2q_v+2-d}|v|_{H^{q_v+1}(\Omega)}^2+{3}\|v\|_{L^\infty(\Omega_k)}^2.
\end{aligned}\]
By \eqref{eq:barq}, we have that $\zeta>d/2$. Consequently,
\begin{equation}\label{eq:bound_v}
\max_{{\bf x} \in \bar\Omega} |\mathfrak v(t,{\bf x})|^2 \leq C\mathcal M(T)\Big(h^{2q_v+2-d} + C_{\rm non}h^{2\zeta-d} + 1\Big) =: m_v \quad \forall t\in [0,T].
\end{equation}
Here $C_{\rm non}$ depends on $\mathcal M(T)$; its precise choice will be specified in \eqref{eq:Cnon}. An argument analogous to that leading to \eqref{eq:bound_v} gives
\begin{equation*}\label{eq:bound_vt}
\max_{{\bf x} \in \bar\Omega} |\dot{\mathfrak v}(t,{\bf x})|^2 \leq m_{\dot v} \quad {\text{for a.e. } t\in [0,T]}.
\end{equation*}

Let $m := \max\{m_v, m_{\dot v}\}$. Recall that the constants defining the admissible coefficient class $\mathcal A_{\rm adm}$ are fixed such that $0<a_\ast < 1 < a^\ast < \infty$, $M_a\in (0,1)$ and $\mu\in (0,4)$. Taking the maximum over $t \in[0, T]$ in \eqref{eq:bound_v} gives ${\|\mathfrak v\|_{L^\infty(0,T;L^\infty(\Omega))}} \leq \sqrt{m}$. Hence, for sufficiently small $\mathcal M(T)$ and $h$ such that $2\bar\beta\sqrt m\le\min\{1-a_*,\,a^*-1\}$, where $\bar\beta = \|\beta\|_{L^\infty(\Omega)}$, we obtain
\[ 0 < a_\ast \leq 1-2\bar \beta \sqrt{m} \leq a_{\mathfrak v}=1-2 \beta \mathfrak v \leq 1+2\bar \beta \sqrt{m} \leq a^\ast < \infty.\]
We next verify the boundedness of $\dot{a}_{\mathfrak v}/a_{\mathfrak v}$. Since $a_{\mathfrak v} = 1-2\beta \mathfrak v$, we have
\[\bigg\|\frac{\dot a_{\mathfrak v}}{a_{\mathfrak v}}\bigg\|_{L^1(0,T;L^{\infty}(\Omega))}\leq \frac{2\bar\beta}{a_\ast}\|\dot{\mathfrak v}\|_{L^1(0,T;L^{\infty}(\Omega))},\quad \bigg\|\frac{\dot a_{\mathfrak v}}{a_{\mathfrak v}}\bigg\|_{L^\infty(0,T;L^{\infty}(\partial\Omega))} \leq \frac{2\bar \beta}{a_\ast}\|\dot{\mathfrak v}\|_{L^\infty(0,T;L^{\infty}(\partial\Omega))}.\]
Thus, by further decreasing $\mathcal M(T)$ and $h$, if necessary, so that $\frac{2\bar\beta T}{a_*}\sqrt m\le M_a$ and $\frac{2\bar\beta}{a_*}\sqrt m\le\mu$, we obtain $\|\dot a_{\mathfrak v}/a_{\mathfrak v}\|_{L^1(0,T;L^{\infty}(\Omega))} \leq M_a < 1$ and $\|\dot a_{\mathfrak v}/a_{\mathfrak v}\|_{L^\infty(0,T;L^{\infty}(\partial\Omega))} \leq \mu < 4$. Hence $a_{\mathfrak v}\in\mathcal A_{\rm adm}$ for every $(\mathfrak u,\mathfrak v)\in\mathfrak B_h$. Consequently, Theorem~\ref{thm2} applies uniformly to every frozen problem associated with an element of $\mathfrak B_h$, with the same constant $C_{\rm lin}$.

With these preliminary estimates established, we carry out the Banach fixed-point argument in three steps. In \textbf{Step 1}, we define the fixed-point map; in \textbf{Step 2}, we show that the map leaves $\mathfrak B_h$ invariant; and in \textbf{Step 3}, we prove that the map is a contraction on $\mathfrak B_h$.

\noindent \textbf{Step 1: the fixed-point map.} For each $(\mathfrak u,\mathfrak v)\in\mathfrak B_h$, where $\mathfrak{B}_h$ is given in \eqref{eq:def_Ball}, we define $\mathcal J(\mathfrak u,\mathfrak v)=(u_h,v_h)$, where \((u_h,v_h)\) is the solution of the linearized semi-discrete scheme \eqref{eq:linearized_scheme}
with coefficient $a_{h}= a_{\mathfrak v}=1-2\beta\mathfrak v \in \mathcal A_{\rm adm}$ and initial data $(u_h^0,v_h^0)$. We note that the operator $\mathcal{J}$ is well-posed thanks to the well-posedness result of Theorem~\ref{thm1}. Moreover, let
\begin{equation*}
\mathcal S_0:=C([0,T];U_h^{q_u})\times C([0,T];V_h^{q_v}), \quad \mathcal X_h^0:=\left\{\begin{aligned}
(w,z) \in  \mathcal S_0: 
(w(0,{\bf x}),z(0, {\bf x}))=(u_h^0, v_h^0);& \\
\int_{\Omega_j} w - w(0,{\bf x})\,d{\bf x} = \int_0^t\int_{\Omega_j}z \,d{\bf x}\,ds \quad \forall \Omega_j\in \mathcal T_h&
\end{aligned}\right\}.
\end{equation*}
For $(w,z)\in \mathcal S_0$, define the lower-order energy seminorm
\[\|(w,z)\|_{\mathcal X_h^0}^2
:=\|\nabla_h w\|_{L^\infty(0,T;L^2(\Omega))}^2
+\|z\|_{L^\infty(0,T;L^2(\Omega))}^2 + \|z\|_{L^2(0,T;L^2(\partial\Omega))}^2.\]
The common initial data and the elementwise averaged relation ensure that this seminorm induces a metric on $\mathcal X_h^0$. Moreover, the elementwise Poincar\'e inequality and the finite-dimensionality of the discrete spaces imply that $(\mathcal X_h^0,\|\cdot\|_{\mathcal X_h^0})$ is a complete metric space. Furthermore, since $\mathfrak B_h \subset \mathcal X_h^0$, the stronger energy bound defining $\mathfrak B_h$, together with the elementwise averaged relation, yields weak-$*$ compactness in the stronger space. The weak-$*$ lower semicontinuity of $\|\cdot\|_{E,h}$ then implies that $\mathfrak B_h$ is closed with respect to the lower-order metric $\|\cdot\|_{\mathcal X_h^0}$. Consequently, $(\mathfrak B_h, \|\cdot\|_{\mathcal X_h^0})$ is a complete metric space. Hence, in the contraction argument, it is sufficient to establish contractivity with respect to the lower-order metric
$\|\cdot\|_{\mathcal X_h^0}$, rather than the stronger energy seminorm $\|\cdot\|_{E,h}$ used to define $\mathfrak B_h$.

\noindent \textbf{Step 2: self-mapping.} We next show that the operator $\mathcal{J}$, $\mathcal J(\mathfrak u,\mathfrak v)=(u_h,v_h)$, is a self-mapping. Since $a_h =a_{\mathfrak v}= 1-2\beta\mathfrak v \in \mathcal A_{\rm adm}$, Theorem~\ref{thm2} and Lemma~\ref{lemma:nonlinear_IC} then yield
\begin{equation}\label{eq:selfmap1}
\begin{aligned}
\|(u - u_h,v - v_h)\|_{E,h}^2
\le{}&
C_{\rm lin}\big( C_{\rm init}h^{2\zeta}\mathcal R_0 + h^{2\zeta}\mathcal{R}(T) + \mathcal Q(T)\big) \\
\leq{}& C_{\rm lin}\big(({C_{\rm init}}+1)h^{2\zeta}\mathcal R(T) + \mathcal Q(T)\big).
\end{aligned}
\end{equation}
To bound $\mathcal Q(T)$, we note that 
\begin{equation}\label{eq:selfmap2}
\|a - a_h\|_{L^\infty(0,T;L^2(\Omega))}^2 +\|a - a_h\|_{L^2(0,T;L^2(\partial\Omega))}^2 \leq 4\bar\beta^2 C_{\rm non}h^{2\zeta}  \mathcal{R}(T),
\end{equation}
and 
\begin{equation}\label{eq:selfmap3}
\|\dot a - \dot a_h\|_{L^\infty(0,T;L^2(\Omega))}^2 +\|\dot a - \dot a_h\|_{L^2(0,T;L^2(\partial\Omega))}^2 \leq 4\bar\beta^2 C_{\rm non}h^{2\zeta}  \mathcal{R}(T).
\end{equation}
Plugging \eqref{eq:selfmap2}--\eqref{eq:selfmap3} into \eqref{eq:selfmap1}, we obtain
\[
\|(u - u_h,v - v_h)\|_{E,h}^2
\leq
C_{\rm lin}h^{2\zeta}\big( {C_{\rm init}}+ 1+8\bar\beta^2C_{\rm non} \mathcal M(T)\big)\mathcal{R}(T).
\]
Thus, for sufficiently small $\mathcal M(T)$ such that $1 - 8\bar\beta^2C_{\rm lin} \mathcal M(T) > 0$, we can choose
\begin{equation}\label{eq:Cnon}
C_{\rm non} := \frac{C_{\rm lin}(1+{C_{\rm init}})}{1- 8\bar\beta^2C_{\rm lin} \mathcal M(T)}
\end{equation}
such that $\|(u - u_h,v - v_h)\|_{E,h}^2
\le{} C_{\rm non}h^{2\zeta} \mathcal{R}(T)$, namely, $(u_h, v_h) \in \mathfrak{B}_h$.

\noindent \textbf{Step 3: contraction.} To establish the strict contractivity of the operator $\mathcal{J}$, let
$(\mathfrak{u}_1,\mathfrak{v}_1)$ and $(\mathfrak{u}_2,\mathfrak{v}_2)$ be arbitrary elements of $\mathfrak{B}_h$. Define the corresponding linearized
coefficients by $a_{h,i} := a_{\mathfrak v_i} = 1 - 2\beta\mathfrak{v}_i \in \mathcal A_{\rm adm}$, and let $(u_{h,i}, v_{h,i}) := \mathcal{J}(\mathfrak{u}_i, \mathfrak{v}_i)$, $i=1,2$, denote the associated solutions of the linearized scheme \eqref{eq:linearized_scheme}. We introduce the differences
\[
    \hat{\mathfrak{u}} := \mathfrak{u}_1 - \mathfrak{u}_2, \quad \hat{\mathfrak{v}} := \mathfrak{v}_1 - \mathfrak{v}_2, \quad \hat{u}_h := u_{h,1} - u_{h,2}, \quad \hat{v}_h := v_{h,1} - v_{h,2}.
\]
Subtracting the two instances of EDG scheme \eqref{eq:linearized_scheme} satisfied by $(u_{h,1},v_{h,1})$ and $(u_{h,2},v_{h,2})$, respectively, yields 
\begin{subequations}\label{eq:diff_scheme}
\begin{align}
\int_{\Omega_j} c^2\nabla \phi_u \cdot \nabla \dot{\hat u}_h\, d{\bf x} + \mathcal K_j(u_{h,1}, v_{h,1};\phi_u;a_{h,1})-\mathcal K_j(u_{h,2}, v_{h,2};\phi_u;a_{h,2}) &= 0,\label{eq:diff_scheme_1}\\
\int_{\Omega_j} \phi_v a_{h,1}\dot{\hat v}_h - 2\beta \phi_v\hat{\mathfrak v}\dot v_{h,2}\, d{\bf x} + \mathcal H_j(u_{h,1},v_{h,1};\phi_v;a_{h,1}) - \mathcal H_j(u_{h,2},v_{h,2};\phi_v;a_{h,2}) &=0,\label{eq:diff_scheme_2}\\
\int_{\Omega_j} \tilde \phi(\dot{\hat{u}}_h - \hat{v}_h)\,d\mathbf{x} &= 0,\label{eq:diff_scheme_3}
\end{align}
\end{subequations}
for all $(\phi_u, \phi_v, \tilde \phi) \in U_h^{q_u}\times V_h^{q_v} \times \mathcal{P}^0$. Invoking the definitions of $\mathcal K$ and $\mathcal H$ in \eqref{eq:defKH_total}, and the definitions of  $\mathcal K_j$ and $\mathcal H_j$ in \eqref{eq:def_KH}, we have that
\begin{equation}\label{eq:diff_aux1}
\begin{aligned}
&\sum_j \mathcal K_j(u_{h,1}, v_{h,1};\phi_u;a_{h,1})-\mathcal K_j(u_{h,2}, v_{h,2};\phi_u;a_{h,2}) \\
= & \ \mathcal K(\hat u_h, \hat v_h;\phi_u;a_{h,1}) - \int_{\partial\Omega}  \frac{(1-\sigma)c}{2}\nabla_h \phi_u\cdot{\bf n}\big(a_{h,2}^{-1/2} - a_{h,1}^{-1/2}\big)\big(c^2\nabla_h u_{h,2} + \alpha \nabla_h v_{h,2}\big)\cdot{\bf n}{\, dS},
\end{aligned}
\end{equation}
and
\begin{equation}\label{eq:diff_aux2}
\begin{aligned}
&\sum_j \mathcal H_j(u_{h,1}, v_{h,1};{\phi_v};a_{h,1})-\mathcal H_j(u_{h,2}, v_{h,2};{\phi_v};a_{h,2}) \\
=&\ \mathcal H(\hat u_h,\hat v_h;\phi_v;a_{h,1}) - \int_{\partial\Omega} \frac{(1-\sigma)\alpha}{2c}\nabla_h\phi_v\cdot{\bf n} \big(a_{h,2}^{-1/2} - a_{h,1}^{-1/2}\big)\big(c^2\nabla_h u_{h,2} + \alpha \nabla_h v_{h,2}\big)\cdot{\bf n}\, dS \\
& \ -\int_{\partial\Omega} \frac{(1+\sigma)c}{2}\phi_v(a_{h,2}^{1/2}-a_{h,1}^{1/2})v_{h,2}\, dS.
\end{aligned}
\end{equation}
Thus, choosing $(\phi_u,\phi_v) = (\hat{u}_h,\hat v_h)$ in \eqref{eq:diff_scheme_1}--\eqref{eq:diff_scheme_2}, summing the resulting equations over all elements, invoking \eqref{eq:diff_aux1}--\eqref{eq:diff_aux2}, and following the derivation in the proof of Theorem~\ref{thm1}, we obtain
\begin{equation}\label{eq:contraction_ineq}
    \begin{aligned}
        &\frac{1}{2}\frac{d}{dt}\int_{\Omega}a_{h,1}|\hat{v}_h|^2+c^2|\nabla_h\hat{u}_h|^2\,d\mathbf{x}+\int_{\Omega}\alpha|\nabla_h\hat{v}_h|^2\,d\mathbf{x}+\mathcal D_I(\hat u_h,\hat v_h) + \mathcal D_{B}(\hat u_h, \hat v_h; a_{h,1})\\
        ={}&\int_{\Omega}\frac{1}{2}\dot{a}_{h,1}|\hat{v}_h|^2+ 2\beta\hat{v}_h\hat{\mathfrak{v}}\dot{v}_{h,2}\,d\mathbf{x} + \mathcal{P}_1 + \mathcal{P}_2,
    \end{aligned}
\end{equation}
where $\mathcal D_I \geq 0$ and $\mathcal D_B \geq 0$ are defined in \eqref{eq:inn} and \eqref{eq:def_DB}, respectively, and
\[\begin{aligned}
&\mathcal{P}_1 = \int_{\partial\Omega} \frac{1+\sigma}{2}c\big(a_{h,2}^{1/2}-a_{h,1}^{1/2}\big)\hat{v}_h v_{h,2}\,dS,\\
&\mathcal{P}_2 = \int_{\partial\Omega}\frac{1-\sigma}{2c}\big(a_{h,2}^{-1/2}-a_{h,1}^{-1/2}\big)\big((c^2\nabla_h u_{h,2}+\alpha\nabla_h v_{h,2})\cdot\mathbf{n}\big)\big((c^2\nabla_h\hat{u}_h+\alpha\nabla_h\hat{v}_h)\cdot\mathbf{n}\big)\,dS.
\end{aligned}\]
To bound $\mathcal P_1 + \mathcal P_2$, since $a_{h,i} \in \mathcal A_{\rm adm}, i=1,2,$ and they are uniformly bounded away from zero, we obtain
\[  
    \big|a_{h,2}^{-1/2} - a_{h,1}^{-1/2}\big| + |a_{h,2}^{1/2} - a_{h,1}^{1/2}| \leq \frac{1}{2}\big(a_\ast^{-3/2} + a_\ast^{-1/2}\big)|a_{h,2} - a_{h,1}| \leq\big(a_\ast^{-3/2} + a_\ast^{-1/2}\big)\bar \beta|\hat{\mathfrak{v}}|.
\]
Consequently, Young's inequality implies
\[  
    \mathcal{P}_1 + \mathcal{P}_2 
    \leq  \frac{1}{2}\mathcal D_{B}(\hat u_h, \hat v_h; a_{h,1})  + C\int_{\partial\Omega}|\hat{\mathfrak v}|^2\Big(|v_{h,2}|^2 + \big|(c^2\nabla_h u_{h,2}+\alpha\nabla_h v_{h,2})\cdot\mathbf{n}\big|^2 \Big)\,dS.
\]
Using this estimate in \eqref{eq:contraction_ineq}, and discarding the nonnegative interelement and volume dissipation terms, yields
\begin{equation}\label{eq:contraction1}
\begin{aligned}
&\hspace{-0.1cm}\frac{d}{dt}\int_\Omega a_{h,1}|\hat v_h|^2 + c^2|\nabla_h\hat u_h|^2\,d{\bf x} + \mathcal D_{B}(\hat u_h, \hat v_h; a_{h,1})\\
\hspace{-0.3cm}\le\ &\Big\|\frac{\dot a_{h,1}}{a_{h,1}}\Big\|_{L^\infty(\Omega)}\!\!\int_\Omega\!\! a_{h,1}|\hat v_h|^2\,d{\bf x}
+ 4\bar\beta\|\dot v_{h,2}(t)\|_{L^\infty(\Omega)}\|\hat v_h\|_{L^2(\Omega)}\|\hat{\mathfrak v}(t)\|_{L^2(\Omega)}
+ C\eta_h(t)\|\hat{\mathfrak v}(t)\|_{L^2(\partial\Omega)}^2,
\end{aligned}
\end{equation}
where $\eta_h(t):=\|v_{h,2}\|_{L^\infty(\partial\Omega)}^2+\big\|\big(c^2\nabla_h u_{h,2}+\alpha\nabla_h v_{h,2}\big)\cdot{\bf n}\,\big\|_{L^\infty(\partial\Omega)}^2$. Solving the differential inequality \eqref{eq:contraction1} gives
\[\|(\hat u_h,\hat v_h)\|_{\mathcal X_h^0}^2 \leq \gamma \|(\hat{\mathfrak u},\hat{\mathfrak v})\|_{\mathcal X_h^0}^2, \qquad \gamma
\lesssim T\|\dot v_{h,2}\|_{L^2(0,T;L^\infty(\Omega))}^2
+
\max_{t\in[0,T]}\eta_h(t).\]
By the estimates established in the preceding steps, the quantities on the right-hand side can be made sufficiently small by imposing the corresponding smallness conditions on $\mathcal M(T)$ and $h$.
In particular, these conditions may be chosen so that $\gamma<1$. Hence, $\mathcal J$ is a strict contraction on $\mathfrak B_h$ with respect to the $\mathcal X_h^0$-metric.

Finally, Banach's contraction principle implies that $\mathcal J$ admits a unique fixed point $(u_h,v_h)\in\mathfrak{B}_h$, i.e., $(u_h,v_h)=\mathcal J(u_h,v_h)$. By the definition of $\mathcal J$, this fixed point is the unique solution of the nonlinear semi-discrete problem \eqref{eq:EDG}--\eqref{eq:mean_value}, subject to the initial data $(u_h(0,{\bf x}), v_h(0,{\bf x})) = (u_h^0,v_h^0)$ obtained through Lemma~\ref{lemma:nonlinear_IC}.
\end{proof}


\begin{remark}\label{rem:abc_weight}
The EDG scheme and the theoretical analysis developed in this work also apply to the alternative absorbing boundary condition considered in \cite{MuhrNikolicWohlmuth2019},
\begin{equation*}\label{eq:abc_alt}
c\sqrt{1-\beta \dot u}\,\dot u
+\big(c^2\nabla u+\alpha\nabla\dot u\big)\cdot{\bf n}=0,
\qquad (t,{\bf x})\in(0,T)\times\partial\Omega.
\end{equation*}
This extension requires only modifications to the boundary-related terms. Specifically, in \eqref{eq:boundary_residual}--\eqref{eq:bccontribution}, the factor $\sqrt{1-2\beta v_h}$ is replaced by $\sqrt{1-\beta v_h}$. Similarly, throughout Sections~\ref{sec:linearized_EDG_scheme} and~\ref{sec:nonlinear}, the boundary weights $\sqrt{a}$, $\sqrt{a_h}$, and $a_{h,i}^{\pm1/2}$ are replaced by $\sqrt{(1+a)/2}$, $\sqrt{(1+a_h)/2}$, and $\big((1+a_{h,i})/2\big)^{\pm1/2}$, respectively, including in the boundary condition of \eqref{eq:westervelt_linear}. The factor $\dot a_h/a_h$ appearing in the physical-boundary estimates is accordingly replaced by $\dot a_h/(1+a_h)$. With these modifications, the stability and convergence results carry over under the same assumptions.
\end{remark}

\section{Numerical studies}\label{sec:numerics}

In this section, we present numerical experiments to illustrate and verify the convergence properties of the EDG scheme introduced in Section~\ref{sec:semidiscrete}. In all experiments, time integration is performed using the classical fourth-order Runge--Kutta method. To ensure that the temporal discretization error does not affect the observed spatial convergence rates, we choose a sufficiently small time step, $\Delta t = h/(10\,q_u^2)$, so that the temporal error is negligible compared with the spatial discretization error. {The Julia implementation used in all numerical experiments presented in this section, along with the scripts and reference data required to reproduce the figures and tables, is publicly available at \href{https://github.com/EDG-Westervelt/EDG_Westervelt}{\mbox{\nolinkurl{github.com/EDG-Westervelt/EDG_Westervelt}}}.}

\subsection{Convergence test}\label{sec:conv_test}
We consider the Westervelt equation \eqref{eq:westervelt} supplemented with a source term $f$ on the right-hand side, chosen as
\[
    f = -\varpi^2 \cos(\varpi t)\big(1 + 2\beta({\bf x})\,\varpi \sin(\varpi t)\,\Phi\big)\Phi
        - \cos(\varpi t)\,\nabla\cdot\big(c^2({\bf x})\nabla\Phi\big)
        + \varpi \sin(\varpi t)\,\nabla\cdot\big(\alpha({\bf x})\nabla\Phi\big),
\]
where $\Phi({\bf x}) = \cos(2\pi x)\sin(2\pi y)$ and $\varpi = 2\sqrt{2}\pi$.
With this choice, the exact solution is $u = \cos(\varpi t)\,\Phi({\bf x})$, which is used to evaluate the numerical errors. All physical quantities in this test are taken to be dimensionless.

The computational domain is the unit square $\Omega = (0,1)^2$. The spatial discretization is based on Cartesian meshes with grid points $(x_k,y_j)=(kh,jh), k,j=0,1,\ldots,N$, where $h = 1/N$ with $N = 5,\ldots,25$. The initial conditions and absorbing boundary data are prescribed consistently with the exact solution. The spatially varying coefficients are chosen as $c({\bf x}) = 1 + 0.1\cos(\pi x)\cos(\pi y)$, $\alpha({\bf x}) = 10^{-5} + 10^{-6}\cos(\pi x)\cos(\pi y)$, and $\beta({\bf x}) = 10^{-4}\big(1 + 0.1\sin(\pi x)\sin(\pi y)\big)$. All simulations are carried out up to the final time $T = 0.1$. For the numerical fluxes \eqref{eq:generalfl} on the interelement boundaries, we consider three choices: the upwind flux (U.-flux), corresponding to $\theta = 1$ and $\lambda_u = \lambda_v = \tau = 1$; the alternating flux (A.-flux), corresponding to $\theta = 1$ and $\lambda_u = \lambda_v = \tau = 0$; and the central flux (C.-flux), corresponding to $\theta = 1/2$ and $\lambda_u = \lambda_v = \tau = 0$. {On the physical boundary we take $\sigma = 0$ in \eqref{eq:flux_bc}.} For the polynomial approximation spaces, we consider both $q_v=q_u$ and $q_v=q_u-1$. Tensor products of Legendre polynomials of degrees $q_u$ and $q_v$ are used as basis functions for $u_h$ and $v_h$, respectively, on the square reference element.

Figure~\ref{fig:conv2d} shows the $L^2$ errors of $u$ and $v$ as functions of the mesh size $h$ for different polynomial degrees. The corresponding convergence rates, computed as the slopes obtained from least-squares fits of the logarithm of the errors against $\log h$, are reported in Table~\ref{tab:conv2d}. Table~\ref{tab:conv2d} also reports the convergence rates for $\|\nabla_h e_u\|_{L^2(\Omega)}$ and $\|e\|_{\rm en}$, where
\begin{equation}\label{eq:numerical_energy}
\|e\|_{\rm en}^2 := \frac{1}{2}\|c\nabla_h e_u\|_{L^2(\Omega)}^2 + \frac{1}{2}\|e_v\|_{L^2(\Omega)}^2.
\end{equation}
To save space, the corresponding error curves for $\|\nabla_h e_u\|_{L^2(\Omega)}$ and $\|e\|_{\rm en}$ are omitted. Several observations can be made from Table~\ref{tab:conv2d}. Overall, the convergence behavior is more stable across the three flux choices when $q_v=q_u-1$ than when the same approximation degree is used for both variables. Nonetheless, these observations are consistent with the existing EDG literature, such as those presented in \cite{AppeloHagstrom2015,appelo2020energy,ZhangHagstromAppelo2019}. Specifically: \vspace{-0.2cm}
\begin{itemize}
\item[(i)] When $q_v=q_u-1$, both the U.-flux and the A.-flux exhibit essentially optimal convergence. Apart from minor deviations for the lowest polynomial degrees, the $L^2$ error of $u$ converges at approximately order $q_u+1$, whereas the $L^2$ error of $v$, the $H^1$-seminorm error of $u$, and the energy error converge at approximately order $q_u$. Thus, using one degree lower for $v_h$ does not lead to a loss of accuracy in the energy-related quantities.\vspace{-0.2cm}
\item[(ii)] For the C.-flux with $q_v=q_u-1$, a clear odd--even behavior is observed. For odd values of $q_u$, the rates are approximately $q_u+1$ for $\|e_u\|_{L^2(\Omega)}$ and $q_u$ for
$\|e_v\|_{L^2(\Omega)}$, $\|\nabla_h e_u\|_{L^2(\Omega)}$, and $\|e\|_{\rm en}$. For even values of $q_u$, however, these rates are reduced by approximately one order, with rates close to $q_u$ for the $L^2$ error of $u$ and $q_u-1$ for the remaining error measures. This parity-dependent behavior is characteristic of the nondissipative central flux in the present experiment.\vspace{-0.2cm}
\item[(iii)] When $q_v=q_u$, the U.-flux continues to produce robust high-order convergence. In particular, the $L^2$ error of $u$ approaches order $q_u+1$, while $\|\nabla_h e_u\|_{L^2(\Omega)}$ is approximately of order $q_u$. The $L^2$ error of $v$ and the energy error also approach order $q_u$, although some deterioration is visible for the lowest approximation degrees.\vspace{-0.2cm}
\item[(iv)] The C.-flux with $q_v=q_u$ again displays a pronounced odd--even effect for the convergence rates of $\|e_v\|_{L^2(\Omega)}$ and $\|e\|_{\rm en}$. In particular, the convergence rates of $\|e_v\|_{L^2(\Omega)}$ and $\|e\|_{\rm en}$ are approximately $q_u$ for odd $q_u$ but approximately $q_u-1$ for even $q_u$. However, the errors in $u$ generally converge at the optimal rate when $q_u\geq 3$, which is consistent with the existing literature.\vspace{-0.2cm}
\item[(v)] The most pronounced loss of accuracy is observed for the A.-flux when $q_v=q_u$. For $q_u\geq 3$, the observed convergence rates are approximately $q_u+1$ for $\|e_u\|_{L^2(\Omega)}$ and $q_u$ for $\|\nabla_h e_u\|_{L^2(\Omega)}$, while those for $\|e_v\|_{L^2(\Omega)}$ and $\|e\|_{\rm en}$ are approximately $q_u-1$. Thus, compared with the choice $q_v=q_u-1$, employing equal-degree approximation spaces with the A.-flux leads to an approximately one-order reduction in the convergence rates of the $v$ and energy errors. These observations are consistent with previous findings in the literature.\vspace{-0.2cm}
\end{itemize}
These results demonstrate that the U.-flux provides the most uniform high-order convergence among the three flux choices considered here. They also indicate that the mixed-degree choice $q_v=q_u-1$ is particularly robust: it yields the expected high-order convergence for both the \mbox{U.-flux} and \mbox{A.-flux} and leads to more consistent behavior across different polynomial degrees. The \mbox{C.-flux} remains convergent but exhibits a distinct parity-dependent loss of accuracy, while the \mbox{A.-flux} is substantially more sensitive to the use of equal-degree approximation spaces.

\begin{figure}[t]
    \centering
    \setlength{\tabcolsep}{0pt}%
    \begin{tabular}{@{}c@{}}
    \includegraphics{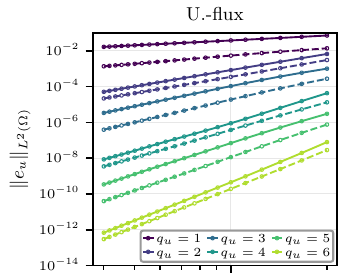}%
    \includegraphics{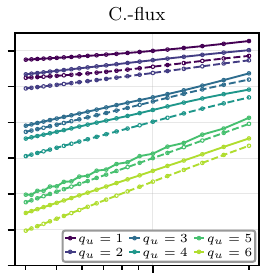}%
    \includegraphics{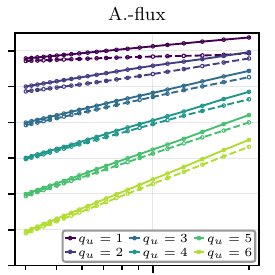}\\
    \includegraphics{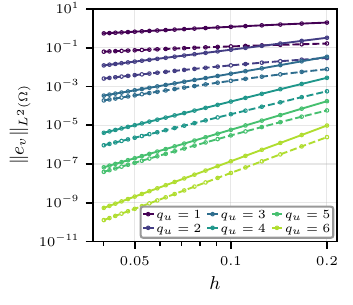}%
    \includegraphics{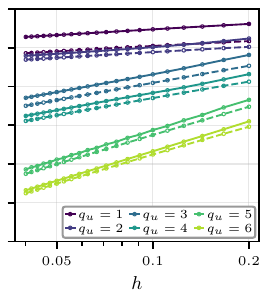}%
    \includegraphics{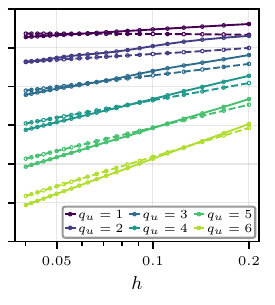}
    \end{tabular}
    \caption{\footnotesize Plots of the $L^2$ errors in $u$ (\textbf{top row}) and $v$ (\textbf{bottom row}) as functions of the mesh size $h$. From left to right, the columns correspond to the upwind, central, and alternating fluxes, respectively. In the legend, $q_u$ denotes the polynomial degree used for the approximation $u_h$. Solid lines correspond to the mixed-degree approximation with $q_v=q_u-1$, whereas dashed lines correspond to the equal-degree approximation with $q_v=q_u$.}
    \label{fig:conv2d}
\end{figure}

\begin{table}[htbp]
\centering
\resizebox{\textwidth}{!}{%
\begin{tabular}{c cccc@{\hspace{1.5em}}cccc@{\hspace{1.5em}}cccc}
\toprule
 & \multicolumn{4}{c}{U.-flux} & \multicolumn{4}{c}{C.-flux} & \multicolumn{4}{c}{A.-flux} \\
\cmidrule(lr){2-5}\cmidrule(lr){6-9}\cmidrule(lr){10-13}
$q_u$ & $\|e_u\|_{L^2}$ & $\|e_v\|_{L^2}$ & $\|\nabla_h e_u\|_{L^2}$ & $\|e\|_{\rm en}$ & $\|e_u\|_{L^2}$ & $\|e_v\|_{L^2}$ & $\|\nabla_h e_u\|_{L^2}$ & $\|e\|_{\rm en}$ & $\|e_u\|_{L^2}$ & $\|e_v\|_{L^2}$ & $\|\nabla_h e_u\|_{L^2}$ & $\|e\|_{\rm en}$ \\
\midrule
\multicolumn{13}{c}{$q_v = q_u - 1$} \\
\midrule
1 & 0.90 & 0.88 & 0.82 & 0.86 & 1.10 & 0.99 & 0.95 & 0.98 & 1.69 & 0.98 & 0.91 & 0.95 \\
2 & 3.08 & 2.03 & 2.02 & 2.02 & 2.00 & 1.07 & 0.98 & 1.00 & 2.97 & 2.07 & 2.13 & 2.08 \\
3 & 3.82 & 2.84 & 2.83 & 2.83 & 4.04 & 3.03 & 3.02 & 3.03 & 4.24 & 2.92 & 3.20 & 2.97 \\
4 & 5.12 & 4.04 & 4.06 & 4.04 & 4.02 & 3.03 & 3.00 & 3.01 & 5.22 & 4.01 & 4.21 & 4.05 \\
5 & 5.87 & 4.86 & 4.85 & 4.86 & 6.05 & 5.03 & 5.04 & 5.04 & 6.25 & 5.01 & 5.22 & 5.05 \\
6 & 7.07 & 6.03 & 6.05 & 6.04 & 5.88 & 5.07 & 4.89 & 5.01 & 7.27 & 6.01 & 6.24 & 6.05 \\
\midrule
\multicolumn{13}{c}{$q_v = q_u$} \\
\midrule
1 & 1.49 & 0.72 & 0.90 & 0.88 & 1.29 & 0.98 & 1.00 & 1.00 & 0.67 & $-0.03$ & 0.36 & 0.03 \\
2 & 3.01 & 1.77 & 2.00 & 1.94 & 2.11 & 0.98 & 1.56 & 1.02 & 2.31 & 0.98 & 1.79 & 1.06 \\
3 & 4.24 & 2.65 & 3.07 & 2.73 & 4.20 & 2.99 & 3.03 & 3.01 & 3.96 & 1.99 & 2.96 & 2.05 \\
4 & 5.08 & 4.01 & 4.00 & 4.01 & 4.87 & 2.99 & 3.87 & 3.02 & 4.93 & 3.02 & 3.93 & 3.05 \\
5 & 6.15 & 4.79 & 5.02 & 4.83 & 6.13 & 5.08 & 5.02 & 5.07 & 5.96 & 4.06 & 4.94 & 4.08 \\
6 & 7.04 & 6.07 & 6.00 & 6.02 & 6.89 & 5.08 & 5.86 & 5.09 & 6.95 & 5.13 & 5.91 & 5.14 \\
\bottomrule
\end{tabular}
}
\caption{\footnotesize Estimated convergence rates for the upwind, central, and alternating fluxes. The \textbf{upper} half of the table corresponds to the mixed-degree approximation with $q_v=q_u-1$, while the \textbf{lower} half corresponds to the equal-degree approximation with $q_v=q_u$.}\label{tab:conv2d}
\end{table}

\subsection{Limiting behavior as \texorpdfstring{$\alpha \rightarrow 0$}{alpha to 0}}
In contrast to the analysis in \cite{AntoniettiEtAl2020}, Theorem~\ref{thm3} remains applicable to the limiting case $\alpha=0$, owing to the dissipation induced by the absorbing boundary condition. It is therefore of interest to investigate numerically how the EDG solution of \eqref{eq:westervelt} approaches the EDG solution with $\alpha=0$ as $\alpha\to0$. In this subsection, we examine the limiting behavior of the numerical solution with respect to the diffusivity parameter $\alpha$.

We consider the Westervelt equation \eqref{eq:westervelt} in $\Omega = (0,1)^2$ with initial data
\begin{equation*}\label{eq:initial_data_experi}
u(0,{\bf x}) = \begin{cases}\exp\Big(1-\dfrac{1}{1-|{\bf x}-{\bf x}_0|^2/R^2}\Big), & |{\bf x}-{\bf x}_0|<R,\\[1mm] 0, & \text{otherwise},\end{cases} \qquad \dot u(0,{\bf x}) = 0,
\end{equation*}
where ${\bf x}_0=(0.5,0.5)$ and $R=0.25$. Since $u(0,\cdot)$ is smooth and vanishes near $\partial\Omega$, these data satisfy the compatibility conditions associated with \eqref{eq:westervelt_2} to any order. The absorbing boundary condition \eqref{eq:westervelt_2} is imposed on $\partial\Omega$. The sound speed $c({\bf x})$ and the nonlinearity coefficient $\beta({\bf x})$ are chosen as in Section~\ref{sec:conv_test}. The problem is discretized using the EDG scheme \eqref{eq:EDG}. We take polynomial degrees $q_u=6$ and $q_v=5$, and employ the upwind flux with $\theta=1$ and $\lambda_u = \lambda_v = \tau = 1$ in \eqref{eq:generalfl} for the interelement boundaries. On the physical boundary we set $\sigma = 0$ in \eqref{eq:flux_bc}. The spatial mesh size is chosen as $h_x=h_y=1/20$. The computation is carried out up to the final time $T=0.5$, so that the waves reach the absorbing boundary. To investigate the limiting behavior as $\alpha\to0$, we perform simulations for $\alpha \in \{10^{-5}, 10^{-6}, 10^{-7}, 10^{-8}, 10^{-9}\}$ and compare the corresponding numerical solutions with the one obtained for $\alpha=0$. Let $e_u^{(\alpha)} := u_h^{(\alpha)}-u_h^{(0)},e_v^{(\alpha)} := v_h^{(\alpha)}-v_h^{(0)}$, where $(u_h^{(\alpha)}, v_h^{(\alpha)})$ denotes the numerical solution corresponding to the parameter $\alpha$, and $(u_h^{(0)},v_h^{(0)})$ denotes the numerical solution for $\alpha=0$. {We measure the difference using $\|e_u^{(\alpha)}\|_{L^2(\Omega)}$ and the energy norm $\|e^{(\alpha)}\|_{\rm en}$ defined in \eqref{eq:numerical_energy}, with $(e_u,e_v)$ replaced by $(e_u^{(\alpha)},e_v^{(\alpha)})$.}
Table~\ref{tab:alpha-order} reports the $L^2$ and energy errors for the different values of $\alpha$, together with the corresponding estimated convergence rates with respect to $\alpha$. As $\alpha$ decreases, both error measures decrease systematically toward zero. In particular, the numerical results exhibit approximately first-order convergence with respect to $\alpha$. This behavior is consistent with the limiting result reported in \cite{GomezMeliani2025}. A rigorous analysis of the vanishing-diffusivity limit for the proposed EDG discretization is beyond the scope of the present work and is left for future investigation.

\begin{table}[htbp]
\centering
\begin{tabular}{c cc cc}
\toprule
$\alpha$ & $\|e_u^{(\alpha)}\|_{L^2}$ & order & $\|e^{(\alpha)}\|_{\rm en}$ & order \\
\midrule
$10^{-5}$ & $7.63\times10^{-5}$ & -- & $4.51\times10^{-3}$ & -- \\
$10^{-6}$ & $7.66\times10^{-6}$ & 1.00 & $4.56\times10^{-4}$ & 0.99 \\
$10^{-7}$ & $7.66\times10^{-7}$ & 1.00 & $4.57\times10^{-5}$ & 1.00 \\
$10^{-8}$ & $7.66\times10^{-8}$ & 1.00 & $4.57\times10^{-6}$ & 1.00 \\
$10^{-9}$ & $7.66\times10^{-9}$ & 1.00 & $4.57\times10^{-7}$ & 1.00\\
\bottomrule
\end{tabular}
\caption{\footnotesize Convergence with respect to $\alpha$ for the upwind flux, corresponding to $\theta=1, \lambda_u=\lambda_v=\tau=1$, with $q_u=6$, $q_v=5$, and $N=20$. The simulation is performed until time $T = 0.5$. For $\alpha>0$, the errors are measured relative to the numerical solution of the limiting problem with $\alpha=0$. The observed rates indicate approximately first-order convergence as $\alpha\to0$.}
\label{tab:alpha-order}
\end{table}

\subsection{Wave steepening}
Finally, to illustrate the influence of the nonlinearity parameter $\beta$ on the wave propagation, we consider the Westervelt equation \eqref{eq:westervelt} on the domain $\Omega=(0,1)^2$ with $c=1$, $\alpha=10^{-6}$, and homogeneous initial data. The wave is generated by the spatially localized and temporally modulated source
\[
    f(t, {\bf x}) = e^{-|{\bf x} - {\bf x}_0|^2/{r_0^2}}\,
    \sin(\varpi_0 t)\, \sin^2(\pi t/T_s), \qquad 0 \le t \le T_s,
\]
and $f(t,{\bf x}) = 0$ for $t>T_s$, where ${\bf x}_0 = (0.15, 0.5)$, ${r_0 = 0.05}$, $\varpi_0 = 16\pi$ and $T_s = 1$. The problem is discretized using the proposed EDG scheme. We take the polynomial degrees $q_u=4$ and $q_v=3$ and, on the interelement boundaries, employ the upwind numerical flux, corresponding to $\theta = 1$ and $\lambda_u = \lambda_v = \tau = 1$ in \eqref{eq:generalfl}. {On the physical boundary we take $\sigma = 0$ in \eqref{eq:flux_bc}.} A uniform Cartesian mesh with $h_x=h_y=1/64$ is used, and the simulation is carried out up to the final time $T=1.8$.

The left panel of Figure~\ref{fig:steepening} shows the computed {$v_h$, which is proportional to the acoustic pressure,} over the whole domain at $t=0.8$, with the concentric wave fronts emitted by the source. The top right panel of Figure~\ref{fig:steepening} compares the numerical solution $v_h$ recorded at the receiver location ${\bf x}_r=(0.7,0.5)$ for the nonlinear Westervelt equation with $\beta=8$ and for the corresponding damped linear wave equation with $\beta=0$. A clear distortion of the waveform is observed in the nonlinear case, with a pronounced steepening of the wave profile relative to the linear solution. This behavior is a characteristic consequence of nonlinear acoustic propagation: portions of the wave associated with larger amplitudes propagate at an effectively different speed, resulting in a progressive deformation and steepening of the waveform.

The bottom right panel of Figure~\ref{fig:steepening} compares the receiver signals for the nonlinear Westervelt equation with $\beta=8$ obtained using the absorbing boundary condition \eqref{eq:westervelt_2} and homogeneous Dirichlet boundary conditions, the latter being considered, for example, in \cite{AntoniettiEtAl2020,GomezMeliani2025}. The absorbing boundary condition employed in the present work allows outgoing waves to leave the computational domain with substantially reduced artificial reflection. In contrast, homogeneous Dirichlet boundary conditions reflect the outgoing waves back into the computational domain, giving rise to pronounced secondary wave components in the receiver signal once the reflected waves reach the observation point. The comparison therefore demonstrates the effectiveness of the absorbing boundary treatment in reducing spurious boundary reflections and, consequently, in providing a more faithful representation of outgoing nonlinear acoustic waves on a truncated computational domain.

\begin{figure}[htbp]
    \centering
    \setlength{\tabcolsep}{0pt}%
    \begin{tabular}{@{}c@{}c@{}}
    \begin{tabular}{@{}c@{}}
    \includegraphics{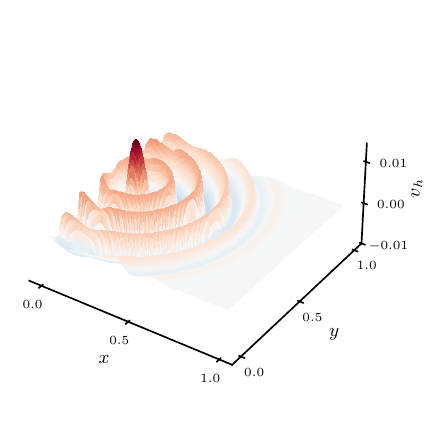}
    \end{tabular} &
    \begin{tabular}{@{}c@{}}
    \includegraphics{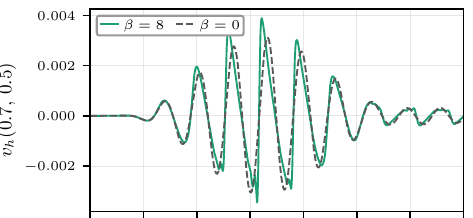}\\
    \includegraphics{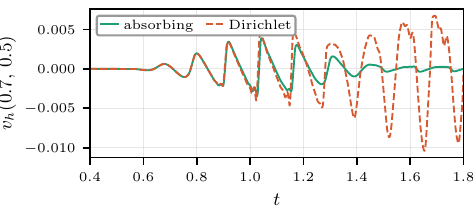}
    \end{tabular}
    \end{tabular}
    \caption{\footnotesize \textbf{Left:} the computed {$v_h$} over the domain at $t=0.8$, for $\beta=8$ with the absorbing boundary condition. \textbf{Right:} receiver signal recorded at $(0.7,0.5)$; top, comparison between nonlinear ($\beta=8$) and linear ($\beta=0$) wave propagation; bottom, comparison, for $\beta=8$, between the absorbing boundary condition and homogeneous Dirichlet boundary conditions.}
    \label{fig:steepening}
\end{figure}

\FloatBarrier
\setcounter{equation}{0}\renewcommand{\theequation}{A.\arabic{equation}}\renewcommand{\theHequation}{A.\arabic{equation}}
\section*{Appendix}

We prove Lemma~\ref{lemma:nonlinear_IC} by the same strategy as in the proof of Theorem~\ref{thm3}. We retain the notation used there whenever possible. For a given
$(\mathfrak u_h^0,\mathfrak v_h^0) \in U_h^{q_u}\times V_h^{q_v}$, define $a_{\mathfrak v}^0:=1-2\beta\mathfrak v_h^0$. We first freeze $a_{\mathfrak v}^0$ and assume that
\begin{equation}\label{eq:bound_on_a}
0<a_* \le a_{\mathfrak v}^0\le a^* < \infty.
\end{equation} 
Consider the following linear problem associated with \eqref{eq:nonlinear_Ritz_IC},
\begin{subequations}\label{eq:frozen_initial_problem}
\begin{align}
\kappa_u\int_{\Omega_j}c^2\nabla\phi_u\cdot\nabla u_h^0\,d{\bf x}+\mathcal K_j(u_h^0,v_h^0;\phi_u;a_{\mathfrak v}^0)
&=\kappa_u\int_{\Omega_j}c^2\nabla\phi_u\cdot\nabla u^0\,d{\bf x}-\int_{\Omega_j}c^2\nabla\phi_u\cdot\nabla v^0\,d{\bf x},\label{eq:frozen_initial_problem_u}\\
\kappa_v\int_{\Omega_j}\phi_v v_h^0\,d{\bf x}+\mathcal H_j(u_h^0,v_h^0;\phi_v;a_{\mathfrak v}^0)&=\kappa_v\int_{\Omega_j}\phi_vv^0\,d{\bf x}-\int_{\Omega_j}\phi_va^0\dot v^0\,d{\bf x},\label{eq:frozen_initial_problem_v}\\
\int_{\Omega_j}(u_h^0-u^0)\,d{\bf x}&=0.
\label{eq:frozen_initial_problem_mean}
\end{align}
\end{subequations}
Let $\mathcal J_0(\mathfrak u_h^0,\mathfrak v_h^0):=(u_h^0,v_h^0)$ denote the solution of \eqref{eq:frozen_initial_problem}. Following the Banach fixed-point argument used in the proof of Theorem~\ref{thm3}, we show that $\mathcal J_0$ admits a unique fixed point in a neighborhood of $(u^0,v^0)$. This fixed point is precisely the solution of \eqref{eq:nonlinear_Ritz_IC}.

\textbf{Well-posedness of $\mathcal J_0$.} For fixed $a_{\mathfrak v}^0$, problem \eqref{eq:frozen_initial_problem} is linear. Define
\[\mathfrak R_{a_{\mathfrak v}^0}
\big((w_h,z_h),(\phi_u,\phi_v)\big):= \kappa_u
\int_\Omega c^2\nabla_hw_h\cdot\nabla_h\phi_u\, d{\bf x} +\kappa_v\int_\Omega z_h \phi_v\, d{\bf x} +
\mathcal K(w_h,z_h;\phi_u;a_{\mathfrak v}^0) + \mathcal H(w_h,z_h;\phi_v;a_{\mathfrak v}^0),\]
where $\mathcal K$ and $\mathcal H$ are defined in \eqref{eq:defKH_total}. Taking $(\phi_u,\phi_v)=(w_h,z_h)$ and arguing as in the proof of Theorem~\ref{thm1}, we obtain
\begin{equation*}
\mathfrak R_{a_{\mathfrak v}^0}
\big((w_h,z_h),(w_h,z_h)\big)=\kappa_u
\|c\nabla_hw_h\|_{L^2(\Omega)}^2+\kappa_v\|z_h\|_{L^2(\Omega)}^2+\|\sqrt\alpha\nabla_hz_h\|_{L^2(\Omega)}^2+\mathcal D_I(w_h,z_h)+\mathcal D_B(w_h,z_h;a_{\mathfrak v}^0),
\end{equation*}
where $\mathcal D_I \geq 0$ and $\mathcal D_B \geq 0$ are defined in \eqref{eq:inn} and \eqref{eq:def_DB}, respectively, and represent the interelement and physical boundary dissipation terms.

Define the initial-data energy seminorm by
\[\|(w,z)\|_{{\rm init},h}^2 := \|\nabla_h w\|_{L^2(\Omega)}^2 + \|z\|_{L^2(\Omega)}^2 + \|z\|_{L^2(\partial\Omega)}^2.\]
Then, $\mathfrak R_{a_{\mathfrak v}^0}
\big((w_h,z_h),(w_h,z_h)\big) \gtrsim \|(w_h,z_h)\|_{{\rm init},h}^2$. Moreover, the mean-value constraint \eqref{eq:frozen_initial_problem_mean} fixes the constant component of $w_h$ on each element. Consequently, $\|\cdot\|_{{\rm init},h}$ defines a norm on the corresponding constrained space. The bilinear form $\mathfrak R_{a_{\mathfrak v}^0}$ is therefore coercive and continuous on this finite-dimensional space. Consequently, \eqref{eq:frozen_initial_problem} admits a unique solution.

\textbf{$\mathcal J_0$ is a self-mapping.} We next define an admissible set incorporating the mean-value constraint, 
\[\mathcal Y_h := \Big\{(w,z)\in U_h^{q_u}\times V_h^{q_v}:\ \int_{\Omega_j}
({w}-u^0)\,d{\bf x}=0
\quad\forall \Omega_j\in\mathcal T_h\Big\},\]
and the neighborhood
\begin{equation}\label{eq:initial_ball}
\mathfrak B_{h,0}:=\Big\{
(\mathfrak u_h^0,\mathfrak v_h^0) \in \mathcal Y_h: \|(\mathfrak u_h^0- u^0,\mathfrak v_h^0- v^0)\|_{{\rm init},h}^2 \le C_{{\rm non},0}h^{2\zeta}\mathcal R_0\Big\}.
\end{equation}
For $(\mathfrak u_h^0,\mathfrak v_h^0)\in\mathfrak B_{h,0}$, the same argument used in deriving \eqref{eq:bound_v} yields
\[\max_{{\bf x} \in \bar\Omega} |\mathfrak v_h^0({\bf x})|^2 \leq C \mathcal M_0\Big(h^{2q_v+2-d} + C_{{\rm non},0}h^{2\zeta -d} + 1\Big) =: m_0.\]
Here $C_{{\rm non},0}$ may depend on $\mathcal M_0$, and its choice will be specified below. By the lower bounds on $\bar q$ in \eqref{eq:barq}, the powers of $h$ are positive. Thus, for sufficiently small $h$ and $\mathcal M_0$, the coefficient $a_{\mathfrak v}^0$ satisfies \eqref{eq:bound_on_a}. Consequently, the frozen problem \eqref{eq:frozen_initial_problem} is uniformly well-posed for every $(\mathfrak u_h^0,\mathfrak v_h^0)\in\mathfrak B_{h,0}$.

We now show that $\mathcal J_0(\mathfrak B_{h,0}) \subseteq \mathfrak B_{h,0}$. Let $e_u^0 := u^0 - u_h^0$ and $e_v^0 :=  v^0 - v_h^0$. Following the argument used in the proof of Theorem~\ref{thm2}, we obtain, with a constant $C_{{\rm lin},0}>0$ independent of $h$ and of $C_{{\rm non},0}$,
\[\|(e_u^0,e_v^0)\|_{{\rm init},h}^2
\leq  \ C_{{\rm lin},0} \Big(h^{2\zeta} \big(|v^0|_{H^{q_v+1}(\Omega)}^2 + |u^0|_{H^{q_u+1}(\Omega)}^2\big) + \int_{\partial\Omega} |a_{\mathfrak v}^0-a^0|^2|v^0|^2\,dS\Big).\]
Since $a_{\mathfrak v}^0-a^0 = -2\beta(\mathfrak v_h^0-v^0)$, it follows from \eqref{eq:initial_ball} that
\[\int_{\partial\Omega} |a_{\mathfrak v}^0-a^0|^2|v^0|^2\,dS
\leq 4\bar\beta^2 \|v^0\|_{L^\infty(\partial\Omega)}^2C_{{\rm non},0}h^{2\zeta}\mathcal R_0.\]
Consequently,
\[\|(e_u^0,e_v^0)\|_{{\rm init},h}^2
\le C_{{\rm lin},0} h^{2\zeta}\big(1+4\bar\beta^2C_{{\rm non},0}\mathcal M_0\big) \mathcal R_0.\]
Thus, if $\mathcal M_0$ is sufficiently small so that $1-4\bar\beta^2 C_{{\rm lin},0} \mathcal M_0 > 0$, we may choose
\begin{equation*}\label{eq:Cnon_0}
C_{{\rm non},0}:= \frac{C_{{\rm lin},0}}{1-4\bar\beta^2 C_{{\rm lin},0} \mathcal M_0}.
\end{equation*}
With this choice, ${\|(u_h^0-u^0,v_h^0-v^0)\|_{{\rm init},h}^2} \leq C_{{\rm non},0} h^{2\zeta} \mathcal R_0$, and hence $(u_h^0,v_h^0)\in \mathfrak B_{h,0}$.

\textbf{$\mathcal J_0$ is a contraction.} Let $(u_{h,i}^0,v_{h,i}^0) = \mathcal J_0 (\mathfrak u_{h,i}^0,\mathfrak v_{h,i}^0), a_{\mathfrak v_i}^0:=1-2\beta\mathfrak v_{h,i}^0, i=1,2$, and introduce $\hat u_h^0:=u_{h,1}^0-u_{h,2}^0, \hat v_h^0:=v_{h,1}^0-v_{h,2}^0$, together with $\hat{\mathfrak u}^0 := \mathfrak u_{h,1}^0-\mathfrak u_{h,2}^0, \hat{\mathfrak v}^0 := \mathfrak v_{h,1}^0-\mathfrak v_{h,2}^0$. Subtracting the two frozen problems \eqref{eq:frozen_initial_problem}, testing with $(\hat u_h^0,\hat v_h^0)$, and arguing as in the derivation of \eqref{eq:contraction1}, we obtain
\[\begin{aligned}
\|c\nabla_h\hat u_h^0\|_{L^2(\Omega)}^2+\|\hat v_h^0\|_{L^2(\Omega)}^2+\mathcal D_B(\hat u_h^0,\hat v_h^0;a_{\mathfrak v_1}^0) \lesssim \Big(\|v_{h,2}^0\|_{L^\infty(\partial\Omega)}^2+ \|(c^2\nabla_h u_{h,2}^0 + \alpha \nabla_h v_{h,2}^0)\cdot{\bf n}\|_{L^\infty(\partial\Omega)}^2\Big)\|\hat{\mathfrak v}^0\|_{L^2(\partial\Omega)}^2.
\end{aligned}\]
Therefore,
\[\|(\hat u_h^0,\hat v_h^0)\|_{{\rm init},h}^2 \leq \gamma_0 \|(\hat{\mathfrak u}^0, \hat{\mathfrak v}^0)\|_{{\rm init},h}^2, \qquad \text{where } \gamma_0 \lesssim \|v_{h,2}^0\|_{L^\infty(\partial\Omega)}^2+ \|(c^2\nabla_h u_{h,2}^0 + \alpha \nabla_h v_{h,2}^0)\cdot{\bf n}\|_{L^\infty(\partial\Omega)}^2.\]
Choosing $\mathcal M_0$ and $h$ sufficiently small so that $\gamma_0<1$,
the mapping $\mathcal J_0$ is a strict contraction on $\mathfrak B_{h,0}$. Banach's contraction theorem therefore yields a unique fixed point 
\begin{equation}\label{eq:initial_aux1}
(u_{h}^0,v_{h}^0) = \mathcal J_0(u_{h}^0,v_{h}^0) \in\mathfrak B_{h,0}.
\end{equation}
At this fixed point, $a_{\mathfrak v}^0 = 1-2\beta v_h^0 = a_h^0$, so that $(u_h^0, v_h^0)$ satisfies \eqref{eq:nonlinear_Ritz_IC}.

We next estimate the initial errors in the time derivatives.

\noindent \textbf{Estimate of the differentiated initial error in $u$.} Initialize the nonlinear EDG scheme by $u_h(0, \cdot)=u_{h}^0, v_h(0,\cdot)=v_{h}^0$. Since the nonlinear initialization \eqref{eq:nonlinear_Ritz_IC} and the EDG scheme \eqref{eq:EDG} involve the same coefficient $a_{h}^0=1-2\beta v_{h}^0$ and the same numerical fluxes at $t=0$, the corresponding spatial terms cancel exactly. Let $\tilde u_h^0, \tilde v_h^0$ denote the projections introduced in \eqref{eq:L2_projection}, and define
\[\tilde e_u^0 := \underbrace{\tilde u_h^0 - u^0}_{\delta_u^0} + \underbrace{u^0 - u_h^0}_{e_u^0},\qquad \tilde e_v^0 := \underbrace{\tilde v_h^0 - v^0}_{\delta_v^0} + \underbrace{v^0 - v_h^0}_{e_v^0}.\] 
Evaluating \eqref{eq:EDG_1} at $t=0$ and subtracting \eqref{eq:nonlinear_Ritz_IC_u} gives
\begin{equation}\label{eq:initial_aux_u}
\int_{\Omega_j} c^2\nabla\phi_u\cdot \nabla(\dot u_h^0-v^0)\,d{\bf x} = \kappa_u \int_{\Omega_j}
c^2\nabla\phi_u\cdot \nabla(u_{h}^0-u^0)\,d{\bf x}.
\end{equation}
Since $\dot{\tilde u}_h^0 = \dot u^0+\dot\delta_u^0 = v^0+\dot\delta_u^0$, we have $\dot{\tilde e}_u^0 = \dot{\tilde u}_h^0-\dot u_h^0 = v^0+\dot\delta_u^0 -\dot u_h^0$, and $u_{h}^0-u^0 = \delta_u^0-\tilde e_u^0$. Substituting these identities into \eqref{eq:initial_aux_u}, taking $\phi_u=\dot{\tilde e}_u^0$, summing over all elements, and using the projection orthogonality following from \eqref{eq:L2_projection}, we obtain
\[\|c\nabla_h\dot{\tilde e}_u^0\|_{L^2(\Omega)}^2 = \kappa_u \int_\Omega c^2\nabla_h\dot{\tilde e}_u^0\cdot \nabla_h \tilde e_u^0\, d{\bf x} \leq \frac{1}{2}\|c\nabla_h\dot{\tilde e}_u^0\|_{L^2(\Omega)}^2 + C\|c\nabla_h \tilde e_u^0\|_{L^2(\Omega)}^2.\]
Thus,
\begin{equation}\label{eq:estimate_on_initialu}
\|\nabla_h\dot{e}_u^0\|_{L^2(\Omega)}^2 \lesssim \|\nabla_h\dot{\tilde e}_u^0\|_{L^2(\Omega)}^2 + h^{2\bar q}{|v^0|_{H^{q_v+1}(\Omega)}^2} \lesssim \|\nabla_h e_u^0\|_{L^2(\Omega)}^2 + h^{2\bar q}\big({|v^0|_{H^{q_v+1}(\Omega)}^2} + |u^0|_{H^{q_u+1}(\Omega)}^2\big).
\end{equation}

\noindent \textbf{Estimate of the differentiated initial error in $v$.} Evaluating \eqref{eq:EDG_2} at $t=0$ and subtracting
\eqref{eq:nonlinear_Ritz_IC_v} gives
\[\int_{\Omega_j} \phi_v \big( a^0\dot v^0 - a_{h}^0\dot v_h^0\big)\,d{\bf x}=\kappa_v\int_{\Omega_j}\phi_v(v^0-v_{h}^0)\,d{\bf x}.\]
Since $\dot v_h^0=\dot v^0+\dot\delta_v^0-\dot{\tilde e}_v^0$ and $v^0-v_{h}^0=\tilde e_v^0-\delta_v^0$, we obtain
\begin{equation}\label{eq:initial_aux2}
\int_{\Omega_j} \phi_v a_h^0 \dot{\tilde e}_v^0 \,d{\bf x} = \int_{\Omega_j} -\phi_v (a^0 - a_h^0)\dot v^0 + \phi_va_h^0 \dot \delta_v^0\,d{\bf x} + \kappa_v \int_{\Omega_j} \phi_v \tilde e_v^0 \, d{\bf x}-\kappa_v \int_{\Omega_j} \phi_v \delta_v^0 \, d{\bf x}.
\end{equation}
We note that the last term vanishes due to the projection orthogonality \eqref{eq:L2_projection}. Taking $\phi_v = \dot{\tilde e}_v^0$ in \eqref{eq:initial_aux2} and summing over all elements gives
\begin{equation*}
\begin{aligned}
\|(a_h^0)^{1/2} \dot{\tilde e}_v^0\|_{L^2(\Omega)}^2 & = \int_{\Omega} -\dot{\tilde e}_v^0 (a^0 - a_h^0)\dot v^0 + \dot{\tilde e}_v^0a_h^0 \dot \delta_v^0\,d{\bf x} + \kappa_v \int_{\Omega} \dot{\tilde e}_v^0 \tilde e_v^0 \, d{\bf x} \\
&\leq \frac{1}{2} \|(a_h^0)^{1/2} \dot{\tilde e}_v^0\|_{L^2(\Omega)}^2 + Ch^{2q_v+2} |\dot v^0|_{H^{q_v+1}(\Omega)}^2+ C\|\tilde e_v^0\|_{L^2(\Omega)}^2+ C\int_\Omega |a^0-a_h^0|^2 |\dot v^0|^2 \, d{\bf x}.
\end{aligned}
\end{equation*}
Thus,
\begin{equation}\label{eq:estimate_on_initialv}
\|\dot e_v^0\|_{L^2(\Omega)}^2 \lesssim \|e_v^0\|_{L^2(\Omega)}^2 + h^{2\bar q}\big({|\dot v^0|_{H^{q_v+1}(\Omega)}^2 + | v^0|_{H^{q_v+1}(\Omega)}^2}\big).
\end{equation}
Combining \eqref{eq:initial_ball}, \eqref{eq:initial_aux1}, and \eqref{eq:estimate_on_initialu}--\eqref{eq:estimate_on_initialv}, we arrive at \eqref{eq:convergence_initial}.

\section*{Acknowledgments}
This work was partially supported by the NSF grants \#DMS-2231482 and \#DMS-2513924, and Simons Foundation through travel support for mathematicians.


\bibliographystyle{abbrv}
\bibliography{reference}

\end{document}